\documentclass[11pt,reqno,a4paper]{article}
\usepackage{amsmath}
\usepackage{amsbsy}
\usepackage{amsthm}
\usepackage{amssymb}
\usepackage{amscd}
\usepackage{mathabx}
\usepackage{accents}
\usepackage{float}
\usepackage{url}

\usepackage{cite}

\usepackage{mathabx}

\usepackage{hyperref}
\hypersetup{
pdftitle={},%
pdfauthor={},%
pdfsubject={},%
pdfkeywords={},%
colorlinks=true,%
linkcolor={black},%
linktoc={},
linktocpage={},%
pageanchor={},
citecolor={black},
}

\usepackage[english]{babel}
\usepackage[utf8]{inputenc}
\usepackage[margin=2.41cm]{geometry}
\usepackage{mathrsfs}
\usepackage[all]{xy}
\usepackage{hyperref}
\usepackage{url}

\usepackage{tikz}
\usetikzlibrary{matrix,arrows,decorations.pathmorphing}

\usepackage{titlesec}
\titleformat{\section}{\large\bfseries\filcenter}{\thesection}{1em}{}
\titleformat{\subsection}{\bfseries}{\thesubsection}{1em}{}

\usepackage{etoolbox}

\makeatletter
\patchcmd{\ttlh@hang}{\parindent\z@}{\parindent\z@\leavevmode}{}{}
\patchcmd{\ttlh@hang}{\noindent}{}{}{}
\makeatother

\newtheorem{thm}{Theorem}[section]
\newtheorem{cor}[thm]{Corollary}
\newtheorem{lemma}[thm]{Lemma}

\newtheorem{prop}[thm]{Proposition}
\newtheorem{ass}[thm]{Assumption}

\newtheorem{conj}[thm]{Conjecture}

\theoremstyle{remark}

\theoremstyle{definition}

\newtheorem{rmk}[thm]{Remark}
\newtheorem{defn}[thm]{Definition}
\newtheorem{exa}[thm]{Example}

\newtheorem{notation}[thm]{Notation}

\numberwithin{equation}{section}
\allowdisplaybreaks[1]

\catcode`@=12

\def\beq{\begin{equation}}
\def\eeq{\end{equation}}

\def\beqn{\begin{equation*}}
\def\eeqn{\end{equation*}}

\def\ben{\begin{enumerate}}
\def\een{\end{enumerate}}

\renewcommand\thanks[1]{%
  \begingroup
  \renewcommand\thefootnote{}\footnote{#1}%
  \addtocounter{footnote}{-1}%
  \endgroup
}

\renewcommand{\tilde}{\widetilde}
\renewcommand{\epsilon}{{\varepsilon}}

\def\stg{{\texttt{STG}}}
\def\aa{{\mathfrak a}}
\def\bb{{\mathfrak b}}

\def\bm{{\mathbf m}}
\def\0{{(0,0)}}
\def\C{{\mathbb C}}
\def\K{{\mathbb K}}
\def\Q{{\mathbb Q}}
\def\F{{\mathbb F}}
\def\N{{\mathbb N}}
\renewcommand\P{{\mathbb P}}
\def\R{{\mathbb R}}
\def\T{{\mathbb T}}
\def\Z{{\mathbb Z}}
\def\bS{\mathbb{S}}

\def\cf{\emph{cf.}~}
\def\ie{\emph{i.e.}~}
\def\lc{\emph{loc.cit.}}
\def\eg{\emph{e.g.}~}

\def\cA{{\mathcal A}}
\def\cE{{\mathcal E}}
\def\cF{{\mathcal F}}
\def\cG{{\mathcal G}}
\def\cH{{\mathcal H}}
\def\cO{{\mathcal O}}
\def\cR{{\mathcal R}}
\def\cS{{\mathcal S}}
\def\cV{{\mathcal V}}
\def\cW{{\mathcal W}}

\def\bH{{\mathbf H}}
\def\bM{{\mathbf M}}
\def\bP{{\mathbf P}}
\def\bR{{\mathbf R}}
\def\fA{{\mathfrak A}}

\def\a{\alpha}

\def\Sp{{\rm Sp \,}}
\def\SO{{\rm SO \,}}
\def\GL{{\rm GL \,}}

\def\ol{\overline}
\def\Id{{\rm Id\,}}
\def\Inv{{\rm Inv\,}}
\def\Sl{{\rm SL\,}}

\def\Aut{{\rm Aut\,}}
\def\Hom{{\rm Hom\,}}
\def\tors{{\rm tors\,}}
\def\free{{\rm free\,}}

\renewcommand{\i}{{\imath}}

\def\Man{{\text{Man}}}

\def\stgAlg{{\stg\text{-Alg}}}

\def\lt{\langle}
\def\gt{\rangle}

\def\bSymp{{\mathbf{Symp}}}

\def\void{{\rm \varnothing}}
\def\then{\Rightarrow}

\def\limpro{\mathop{\lim\limits_{\displaystyle\leftarrow}}}

\usepackage[normalem]{ulem}

\begin{document}

\begin{flushright}

\baselineskip=4pt

\end{flushright}

\begin{center}
\vspace{5mm}

{\Large\bf ON THE UNIQUENESS OF INVARIANT STATES}


\vspace{5mm}

{\bf by}

\vspace{5mm}

 { \bf Federico Bambozzi}\\[1mm]
\noindent  {\it Mathematical Institute, Radcliffe Observatory Quarter, Woodstock
Road, Oxford OX2 6GG, UK}\\[1mm]
email: \ {\tt bambozzif@maths.ox.ac.uk}
\\[6mm]
{  \bf Simone Murro}\\[1mm]
\noindent  {\it Mathematisches Institut, }{\it Universit\"at Freiburg, } {\it D-79104 Freiburg, Germany}\\[1mm]
email: \ {\tt simone.murro@math.uni-freiburg.de}
\\[8mm]
\today
\\[10mm]
\end{center}

\begin{abstract}
Given an abelian group $\cG$ endowed with a $\T=\R/\Z$-pre-symplectic form, we assign to it a symplectic twisted group $*$-algebra $\cW_\cG$ and then we provide criteria for the uniqueness of states invariant under the ergodic action of the symplectic group of automorphism.  As an application, we discuss the notion of natural states in quantum abelian Chern-Simons theory. 
\end{abstract}

\\[0.5mm]

\tableofcontents

\renewcommand{\thefootnote}{\arabic{footnote}}
\setcounter{footnote}{0}

\section{Introduction}

The study of groups of automorphisms
of operator algebras and states invariant
under the action of the group automorphisms has been and continues to be at the forefront of mathematical research.
For sake of completeness let us recall two motivations that led us to study invariant states.

\paragraph*{Invariant states in statistical mechanics.}
Most applications of the algebraic approach to statistical mechanics 
have focused on equilibrium phenomena and have aimed to
justify the interpretation of equilibrium states as states over a $C^*$-algebra. 
During the years, two different but related type of analysis has been developed. The first one begins with a specific Hamiltonian operator $H_V$ which incorporates a description of the interactions and boundary
conditions for particles in a finite region $V$ and then tries to construct the
Gibbs equilibrium state
$$\omega_\beta:=\lim_{V\to\infty} \omega_{V,\beta}(\aa)=\frac{\textnormal{Tr}(e^{-\beta H_V}\aa)}{\textnormal{Tr} e^{-\beta H_V}}$$
where $\beta$ is the inverse temperature of the system. Clearly, for this construction to work, we need $H_V$ to be a self-adjoint operator and $e^{-\beta H_V}$ has to be trace-class, a condition which can often be established in explicit models. Unfortunately, this type of analysis is not well-suited for a quantum field propagating in a
given gravitational background field. In these cases one immediately encounters three well-known
problems: in a general curved spacetime there is no clear notion of particle, no clear choice of a
Hamiltonian operator and, even if there were,
the operator $e^{-\beta H_V}$ is not trace-class. 
In the second type of analysis, the starting point is a $C^*$-dynamical system, namely a $C^*$-algebra $\fA$ together with a one-parameter group of $*$-automorphisms $ t\in G \mapsto \Phi_t$ on $\fA$. From a physical perspective,
the algebra $\fA$ is interpreted as the algebra of observables and $\Phi_t$ describes the time evolution of the system. In this setting, an equilibrium state is defined to be any linear, positive and normalized functional $\omega:\fA\to\C$ which is invariant under the action of $\Phi_t$, namely 
$$ \Phi_t^*(\omega):= \omega \circ \Phi_t = \omega   \qquad\text{for any $t\in G$\,.} $$
These two types of analysis are not unrelated: As noticed by Kubo~\cite{Kubo} and later by Martin and
Schwinger \cite{MaSc}, Gibbs states can be rewritten as an equilibrium state for some suitable $C^*$-dynamical system and they can be fully characterized by a property now known as \emph{KMS$_\beta$ condition}. This property was proposed in~\cite{HHW} as a defining characteristic for thermal equilibrium states, even when the Gibbs state is no longer
defined, and it clearly depends on a parameter $\beta$ which again represents the inverse temperature. Let us remark, that KMS-states are in general not unique. For interesting systems with interaction, one expects that for large temperatures (\textit{i.e. }small $\beta$), the disorder will be predominant so that there will exist
only one KMS-state. Instead, for low enough temperatures (\textit{i.e. }large $\beta$), some order should set in and allow for the existence of various thermodynamical phases, namely various KMS-states.  
 For further details on Gibbs and KMS states, we refer to~\cite{OA1,AQFT2} and we also recommend the recent review~\cite{KO}. 

\paragraph*{Invariant states and ergodic theory.}
In its simplest form, the classical pointwise ergodic theorem of Birkhoff asserts that if $(X,\Sigma, \mu)$ is a probability measure space, $\Phi$ is an
invertible measure-preserving transformation of $X$ and $f$ is an integrable complex valued function on $X$, then the averages
$$s_n(f)=\frac{1}{n}\sum_{j=0}^{n-1} \Phi^j f$$
converge almost everywhere to an $\Phi$-invariant function $\hat{f}$ (where $\Phi f$ is the function defined by $\Phi f(x)=f(\Phi (x))$).
If we restrict ourself to the case where $f$ is
bounded, we are dealing with an element of the commutative von Neumann algebra $\fA=L^\infty (X,\mu)$, and $\Phi$ gives rise to an automorphism of $\fA$. So the Birkhoff's ergodic theorem can be stated in terms of von Neumann algebra: If $T$ is an automorphism of a commutative von Neumann algebra $\fA$ and there exists a faithful $\Phi$-invariant normal state of $\fA$ then for any $\aa$ in $\fA$ the averages $s_n(\aa)$ converge ``almost uniformly'' to a $\Phi$-invariant element $\hat{\aa}$. The generalization to noncommutative von Neumann algebras was initiated by Kov\'acs and Sz\"ucs in~\cite{KovacsSzucs} and then completed by Lance in~\cite{Lance}. It is clear that the analysis of invariant (normal) states over a $*$-algebra is of crucial importance for noncommutative ergodic theory. For further details on operator algebras and ergodic theory, we refer to\cite{ergodicbook,NesSto}. \\

Despite being so important, having a general criterion for the uniqueness of invariant states seems still to be too far reaching. However, nice and profound results can be obtained if the action of the group automorphisms is ergodic, \textit{i.e.}  the fixed point algebra is trivial.
 In \cite{KovacsSzucs} it was shown that if the action of a group of automorphism $\cE$ on a von Neumann algebra $\cR$ is ergodic then there exists a unique $\cE$-invariant state $\omega$ over $\cR$.
 Ten years later, St\o rmer showed in~\cite{Stormer} that, under the additional assumption that $\mathcal{E}$ is a locally compact abelian group, the unique $\mathcal{E}$-invariant state is a tracial state. 
 For several years it has been an open problem if the same results hold with weaker assumptions (\cf \cite{OPK}) and a positive answer was given in \cite{HLS}, where it was shown that if $\mathcal{E}$ is a compact ergodic group of automorphism acting on an unital $C^*$-algebra, the unique $\mathcal{E}$-invariant state is a trace. 
 Let us underline that the existence and uniqueness of $\mathcal{E}$-invariant states on a $C^*$-algebra (which is not necessary a Von Neumann algebra) is guaranteed by the fact that $\mathcal{E}$ is compact, see \cite{stormer1}.
 
 In most of the models inspired by mathematical physics, the group of $*$-automorphisms is not compact nor abelian, even if it acts ergodically. For example in abelian Chern-Simons theory (as studied in~\cite{DMS}) the automorphisms group $\mathcal{E}$ coincides with the symplectic group of automorphisms. Hence, a general criterion for the existence and uniqueness of invariant states is still missing. The first step in this direction was done in \cite{BMP}, where invariant states by the symplectic group of automorphisms on group algebras with involutions that define irrational non-commutative tori have been classified using elementary, and mostly algebraic, methods. More precisely, it was shown that for irrational rotational algebras the only state invariant under the action of the symplectic group is the canonical trace state. These $*$-algebras are obtained as twisted group algebras for the abelian groups $\Z^{2n}$ equipped with a symplectic form and the action of the symplectic group $\Sp(2n,\Z)$ on $\Z^{2n}$ can be lifted to an ergodic group of $*$-automorphisms.
     
The aim of this paper is to generalize the result of~\cite{BMP} to a broader class symplectic twisted group algebras. To this end, in Section \ref{Sec: SAB} we introduce the abstract notion of a $\cR$-pre-symplectic abelian group, \ie an abelian group equipped with a $\cR$-pre-symplectic form, being $\cR$ a fixed abelian group where the pre-symplectic form takes values. This notion is considered for encompassing in a single concept different examples of symplectic forms arising in different contexts. 
From Section~\ref{sec:group algebra}, we will restrict our attention to the case $\cR = \R/\Z$ (that we denote $\T$). After recalling how to assign a twisted group $*$-algebra $\cW_\cG$ to a $\T$-pre-symplectic abelian group $(\cG, \sigma)$, where $\sigma$ denotes  a fixed pre-symplectic form, in Section~\ref{sec:invariant states} we study $\Sp(\cG)$-invariant states on $\cW_\cG$ (where $\Sp(\cG)$ is the symplectic group of $\cG$) for some specific $\cG$. A summary of the main results obtained is the following (see also Theorem \ref{thm:sum_up}):
\begin{thm} \label{thm:introduction}
Let $(\cG, \sigma)$ be a $\T$-pre-symplectic abelian group, then
\begin{enumerate}
\item if $\sigma$ is degenerate, then $\cW_\cG$ admits plenty of $\Sp(\cG)$-invariant states;
\item if $(\cG, \sigma)$ is symplectic, irrational (in the sense of Definition \ref{defn:irrational}), the symplectic form is diagonalizable (in the sense of Notation \ref{not: diagonalization}) and $(\cG, \sigma)$ satisfies a technical assumption (specified in Theorem \ref{thm:NCtori_no_torsion}), then $\cG$ is torsion-free and the canonical trace is the unique $\Sp(\cG)$-invariant state on $\cW_\cG$; 
\item  if $(\cG, \sigma) \cong \bigoplus_{i \in I} ((\Z/n\Z)^2, \sigma_2)$, where $I$ has infinite cardinality and $\sigma_2$ is the canonical symplectic form, then the canonical trace is the unique $\Sp(\cG)$-invariant state on $\cW_\cG$.
\end{enumerate}
\end{thm}

Notice that in point 2. of Theorem \ref{thm:introduction} it is not assumed that $\cG$ is finitely generated (otherwise the theorem reduces to the result of \cite{BMP}) and neither that it is a free module over $\Z$. We conclude the paper with some conjectures and by giving an application of our results to abelian Chern-Simons theory.

\subsection*{{Acknowledgements}}

We would like to thank Sebastiano Carpi, Claudio Dappiaggi, Nicol\'o Drago, Francesco Fidaleo, Kobi Kremnizer and Vincenzo Morinelli for helpful discussions related to the topic of this paper. We are grateful to the referees for useful comments on the manuscript.

\subsection*{{Funding}}
F. B. is supported the DFG research grant BA 6560/1-1 ``Derived geometry and arithmetic'' and he thanks the Mathematical Institute of the University of Freiburg for the kind hospitality during the preparation of this work as well as the DFG GRK 1821 ``Cohomological Methods in Geometry''   for supporting this stay. S. M. is supported by the research grant ``Geometric boundary value problems for the Dirac operator'' and partially supported by the DFG research training group GRK 1821 ``Cohomological Methods in Geometry''.

\subsection*{Notation and conventions}
\begin{itemize}
\item[-] The set of prime numbers is denoted by $\P$.
\item[-] For any $p\in \P$,  we denote with $\F_p$ the finite field $ \Z / p \Z$, with $\Q_p$ the field of $p$-adic numbers.
\item[-] The symbol $\K$ denotes one of the elements of the set $\{\Z,\Q,\R,\C,\F_p,\Q_p\}$.
\item[-] With $\T$ we denote the torus $\R/\Z$.
\item[-] $\cG$ and $\cR$ denote abelian groups.
\item[-] $\sigma_{2 n}$ denotes the \emph{canonical symplectic form} as defined in Example \ref{exa:prototypical}.
\end{itemize}

\section{Symplectic abelian groups}\label{Sec: SAB}

A \emph{$\cR$-symplectic form} on $\cG$ is a map of abelian groups $\sigma: \cG \times \cG \to \cR$ that satisfies the following properties:
\begin{enumerate}
\item[(I)] \emph{Bilinearity:} For any $x,y,x',y' \in \cG$, it holds that
$$\sigma(x + y, x' + y') = \sigma(x, x') + \sigma(x, y') + \sigma(y, x') + \sigma(y, y').$$ 
\item[(II)] \emph{Skew-symmetricity:} For all $x \in \cG$, it holds that $\sigma(x,x) = 0$.
\item[(III)] \emph{Non-degeneracy}:   If $\sigma(x,y) = 0$ for all $y \in \cG$, then $x = 0$.
\end{enumerate}

If we drop the requirement that $\sigma$ is non-degenerate we say that the form is \emph{$\cR$-pre-symplectic}.
By denoting with $0_\cG$ the unit in $\cG$, it is easy to see that for all $x,y\in\cG$, a $\cR$-pre-symplectic form satisfy the relations 
$$ \sigma(0_\cG, x) = 0, \qquad \sigma(x, y) = - \sigma(y, x)\,.$$

\begin{defn} \label{defn:symplectic_abelian_group}
A \emph{$\cR$-(pre-)symplectic abelian group} is a pair consisting of an abelian group $\cG$ together with a $\cR$-(pre-)symplectic form $\sigma:\cG\times\cG\to\cR$.
\end{defn}

\begin{rmk} \label{rmk:quasi_symplectic_groups}
We remark that in the case when $\cG$ is a  discrete abelian group and $\cR=\T$, our definition of $\cR$-symplectic abelian group coincides with the definition of  \emph{quasi-symplectic spaces} given in~\cite{LP}.
\end{rmk}

Clearly, Definition~\ref{defn:symplectic_abelian_group} is not the most general one. Namely, it is possible to generalize it  by considering pairs consisting of a module ${\tt M}$ over a base ring ${\tt R}$ and a skew-symmetric, non-degenerate ${\tt R}$-bilinear form on ${\tt M}$ with values on a fixed ${\tt R}$-module. We are not interested in developing this version of the theory in this work, although it can be done by easily adapting our discussion.
Before proving some properties of $\cR$-symplectic abelian groups, we give some examples which are present in the literature.

\begin{exa} \label{exa:prototypical}
The prototypical example of a $\cR$-symplectic abelian group is obtained by endowing $\K^{2 n}$ with the standard $\K$-symplectic form $\sigma: \K^{2 n}\times \K^{2 n} \to \K$ (in this case $\cR = \K$ as an abelian group). After fixing a base for $\K^{2n}$ and the canonical scalar product $<\cdot\,,\cdot>_{\K^{2n}}$, the $\K$-symplectic forms read
$$ \sigma(x,y)=<x,\sigma_{2n} \,y>_{\K^{2n}}
\qquad \text{with}\qquad
\sigma_{2n}=\begin{pmatrix}
0 & \Id_n \\
-\Id_n & 0
\end{pmatrix}
$$
and $\Id_n$ the $n\times n$ identity matrix. The $\K$-symplectic abelian groups have been used in several topics, e.g. for $\K=\Z$ in noncommutative geometry~\cite{BMP,DegliEspositi} and  in abelian Chern-Simons theory~\cite{DMS}, for $\K=\R,\C$ in quantum mechanics~ \cite{AS,FV}, in symplectic geometry and deformation quantization~\cite{defquant}, for $\K=\F_p$ in modal quantum theory~\cite{SW1,SW2} and  for $\K=\Q_p$ in $p$-adic quantum mechanics~\cite{Zel}. 
\end{exa}

\begin{notation}
With a slight abuse of notation, we refer to $\sigma_{2n}$ as the \emph{canonical symplectic form}.
\end{notation}

Arithmetic geometry is another source of examples of symplectic abelian groups.

\begin{exa} \label{exa:weil_pairing}
The Tate modules associated to an elliptic curve equipped with the \emph{Weil pairing} are $\T$-symplectic abelian groups. Explicitly, let $E$ be an elliptic curve defined over $\C$ and let $l \in \N$. The $l$-torsion points of $E$ are denoted by $E[l]$ and it is well known that $E[l] \cong (\Z/l\Z)^2$. The Weil pairing on $E[l]$ is defined by the map
\[ e_l: E[l] \times E[l] \to \T \]
given by
\[ { e_l((a,b),(c,d)) = e^{2 \pi \i \frac{(ad - bc)}{l}}}. \]
It is easy to check that $e_l$ is bilinear, non-degenerate and skew-symmetric. Notice also that $e_{m l}$ is compatible with $e_l$, for any $m \in \N$, in the sense that 
\[ e_{m l}((a,b),(c,d)) = e_l ((m a, m b),(c, d)) = e_l ((a, b),(m c, m d)). \]
Hence, fixing a prime $p$, one can pass to the limit and obtain a Weil paring
\[ e_p:  \limpro_{n} E[p^n] \times \limpro_{n} E[p^n] \to \T \]
on the Tate modules $T_p(E) = \underset{n}\limpro E[p^n]$.
\end{exa}

As we shall see in the next example, not every symplectic abelian group appearing in natural examples is finite-dimensional.
\begin{exa} \label{exa:QFT}
Consider a real (resp. complex) vector bundle $E$ over an oriented manifold $M$ and denote the space of compactly supported sections with $\Gamma_c(E)$.
Then  $\big(\Gamma_c(E) \oplus \Gamma_c(E),\sigma\big)$  forms a (resp. $\C$-)$\R$-symplectic abelian group, where $\sigma$ is given by 
\begin{equation}\label{eq:sigma_QFT}
\sigma(f,g):=\frac{\int_M (f ^T \sigma_{2n}\,\, g )(x) \,  vol_M}{\int_M vol_M}
\end{equation}
being $\sigma_{2n}$ the canonical symplectic form and  $vol_M$ the volume form of $M$. These  symplectic abelian groups  have been intensively used to quantize field theory on Lorentzian manifolds -- see e.g.~\cite{BDH, gerard,FR16,ThomasAlex,MurroVolpe,simogravity,
simo3,simo2,simo1,simo4} for reviews or textbooks.
\end{exa}

By defining a \emph{morphism} $\phi: (\cG_1, \sigma_{\cG_1}) \to (\cG_2, \sigma_{\cG_2})$ to be a group homomorphism of the underlying groups that preserves the values of the (pre-)symplectic forms, the class of $\cR$-(pre-)symplectic abelian groups forms a category. We make this concept more precise in the next definition.

\begin{defn} \label{defn:cateogry_symplectic_groups}
    We denote by $\bSymp_\cR$ (resp. $\bP\bSymp_\cR$) the category whose objects are $\cR$-symplectic (resp. $\cR$-pre-symplectic) abelian groups $(\cG_i, \sigma_{\cG_i})$ and whose morphisms are group homomorphisms $\phi: \cG_1 \to \cG_2$ for which the following diagram commutes
\[
\begin{tikzpicture}
\matrix(m)[matrix of math nodes,
row sep=2.6em, column sep=2.8em,
text height=1.5ex, text depth=0.25ex]
{ \cG_1 \otimes_\Z \cG_1   & \cR   \\
  \cG_2 \otimes_\Z \cG_2   &  \\};
\path[->,font=\scriptsize]
(m-1-1) edge node[left] {$\phi \otimes \phi$} (m-2-1);
\path[->,font=\scriptsize]
(m-1-1) edge node[auto] {$\sigma_{\cG_1}$} (m-1-2);
\path[->,font=\scriptsize]
(m-2-1) edge node[right] {$\sigma_{\cG_2}$} (m-1-2);
\end{tikzpicture}.
\]
\end{defn}

\begin{rmk}
The category of abelian groups fully faithfully embeds in $\bP\bSymp_\cR$ via the functor that associates to an abelian group $\cG$ the $\cR$-pre-symplectic abelian  group $(\cG, \sigma_0)$ where $\sigma_0$ is the trivial pre-symplectic form that  has value identically $0$, nevertheless $\bP\bSymp_\cR$ is not additive because the sum of two morphisms does not  preserve the values of the pre-symplectic form, in general.
\end{rmk}

In analogy with symplectic geometry, for any subset $A \subset \cG$ of a $\cR$-pre-symplectic abelian group we define the \emph{orthogonal} subset as
\[ A^\perp = \{ x \in \cG | \sigma(x, a) = 0, \forall a \in A \}. \]

We will use the following simple construction for producing new $\cR$-pre-symplectic abelian groups from known ones.

\begin{prop} \label{prop:direct_sum}
    Let $\{ (\cG_i, \sigma_i) \}_{i \in I}$ be a (small) family of objects of $\bP\bSymp_\cR$. Then, the direct sum (\ie coproduct) of this family is the $\cR$-pre-symplectic abelian group given by
    \[ \cG = \bigoplus_{i \in I}(\cG_i, \sigma_{\cG_i}) = \Big(\bigoplus_{i \in I} \cG_i, \sum_{i \in I} \sigma_{\cG_i} \Big). \]
    If all $\{ (\cG_i, \sigma_{\cG_i}) \}_{i \in I}$ are $\cR$-symplectic abelian groups then also the direct sum is a $\cR$-symplectic abelian group.
\end{prop}
\begin{proof}
Using the fact that the tensor product functor commutes with colimits, as it is a left adjoint functor, we can immediately conclude that
\[ \cG \otimes_\Z \cG \cong \bigoplus_{(i,j) \in I \times I} \cG_i \otimes_\Z \cG_j,\]
as plain abelian groups. Moreover, since any element of $\displaystyle{\cG = \bigoplus_{i \in I} \cG_i}$ is zero in all but finitely many components, we can define $\displaystyle{\sum_{i \in I} \sigma_{\cG_i}}$ in the following way: For any $g = (g_i)_{i \in I}, h = (h_i)_{i \in I} \in \cG$ 
\[ (\sum_{i \in I} \sigma_{\cG_i})(g,h) = \sum_{i \in I} \sigma_{\cG_i}(g_i, h_i). \]
The latter sum is always finite and therefore well-defined. 

By the universal property of $\cG$, there are canonical morphisms $\iota_i: \cG_i \to \cG$ which induce morphisms $\iota_i \otimes \iota_i: \cG_i \otimes \cG_i \to \cG \otimes \cG$ by functoriality. We need  to check that $\iota_i$ are morphisms of $\cR$-pre-symplectic abelian groups. This amounts to show that the diagram
    \[
    \begin{tikzpicture}
    \matrix(m)[matrix of math nodes,
    row sep=2.6em, column sep=2.8em,
    text height=1.5ex, text depth=0.25ex]
    { \cG_i \otimes \cG_i   & \cR   \\
      \cG \otimes \cG    &  \\};
    \path[->,font=\scriptsize]
    (m-1-1) edge node[left] {$\iota_i \otimes \iota_i$} (m-2-1);
    \path[->,font=\scriptsize]
    (m-1-1) edge node[auto] {$\sigma_{\cG_i}$} (m-1-2);
    \path[->,font=\scriptsize]
    (m-2-1) edge node[right] {$\underset{i \in I}\sum \sigma_{\cG_i}$} (m-1-2);
    \end{tikzpicture}
    \]
    is commutative. But this follows immediately by the definition of $\displaystyle{\sum_{i \in I} \sigma_{\cG_i}}$ and by the fact that its restriction on the image of $\iota_i$ is equal to $\sigma_{\cG_i}$. Hence, it follows easily that $\cG$ has the universal property of the coproduct in $\bP\bSymp_\cR$. It is  also immediate to check that if $\{ (\cG_i, \sigma_{\cG_i}) \}_{i \in I}$ are $\cR$-symplectic abelian groups then also $\cG$ is a $\cR$-symplectic abelian group.
\end{proof}

Given a $\cR$-symplectic abelian group, it is important to study subgroups which are compatible with the symplectic structure, in the sense that non-degeneracy is preserved. This leads to the following definition.

\begin{defn} \label{defn:sub_symplectic_abelian_group}
    Let $(\cG, \sigma)$ be a $\cR$-symplectic abelian group. A subgroup $\cH \subset \cG$ is called \emph{$\cR$-symplectic abelian subgroup} if the symplectic form $\sigma$ restricts to a symplectic form on $\cH$.
\end{defn}

Definition \ref{defn:sub_symplectic_abelian_group} is less obvious than one expects at first glance. Indeed, the restriction of a symplectic form to a subgroup of $\cG$ does not preserve, in general, the non-degeneracy property of the form.
\begin{prop} \label{prop:hyperbolic plane}
    Let $(\cG, \sigma)$ be a $\cR$-symplectic abelian group and $x \in \cG$. Suppose that $\cR$ is torsion free, then there exists a $y \in \cG$ such that the abelian subgroup of $\cG$ generated by $x$ and $y$ is a $\cR$-symplectic abelian subgroup of $\cG$.
\end{prop}
\begin{proof}
    We denote by $\lt x, y \gt \subset \cG$ the subgroup of $\cG$ generated by $x$ and $y$. Since $\sigma$ is non-degenerate, by definition for any $x \in  {\cG}$ there exists a $y$ such that $\sigma(x, y) \ne 0$. Consider such a $y \in \cG$. Then, any element in $z \in \lt x, y \gt$ can be written as 
    \[ z = n_x x + n_y y, \ \ \ n_x, n_y \in \Z \]
    and for any such, non-null, element we have to find a $z' \in \lt x, y \gt$ such that
    \[ \sigma(z, z') \ne 0. \]
    By writing $z' = m_x x + m_y y$ and using bilinearity we get
    \[ \sigma(n_x x + n_y y, m_x x + m_y y) = n_x m_x \sigma(x,x) + n_x m_y \sigma(x,y) + n_y m_x \sigma(y,x) + n_y m_y \sigma(y,y). \]
    As by hypothesis $\sigma(x,y) \ne 0$ and $\cR$ is torsion free, one can choose $m_y = 1$ and $m_x = 0$ so that
    \[ \sigma(n_x x + n_y y, m_x x + m_y y) =  n_x \sigma(x,y) \ne 0 \] 
    because $\cR$ is torsion free.
\end{proof}
Using the jargon of symplectic geometry, Proposition \ref{prop:hyperbolic plane} can be restated by saying that if $\cR$ is torsion free then each $x \in \cG$ is contained in a hyperbolic plane. Therefore, we give the following definition.

\begin{defn} \label{defn;hyperbolic_plane}
We call \emph{hyperbolic plane $\cH$} of $\cG$ any $\cR$-symplectic subgroup of $\cG$ given as in Proposition \ref{prop:hyperbolic plane}.
\end{defn}

\begin{rmk} \label{rmk:rank_hyperbolic_plane}
Notice that any hyperbolic plane $\cH \subset \cG$ is a finitely generated abelian group of rank $2$. It can be easily shown that if $\cG$ is torsion-free it is isomorphic to $\Z^2_r = (\Z^2, r \sigma_2)$, where for any $r\in\cR$ we use the notation
\[ r \sigma_2 = \begin{pmatrix}
0 & r \\
-r  & 0 
\end{pmatrix}. \]
\end{rmk}
Our notation in Remark~\ref{rmk:rank_hyperbolic_plane} is well-defined due to the fact that for $x_1, x_2 \in \cG$  the products
\[ \begin{pmatrix}
0 & r \\
-r  & 0 
\end{pmatrix}
\begin{pmatrix}
x_1 \\
x_2 
\end{pmatrix} =\begin{pmatrix}
-r x_2 \\
rx_1 
\end{pmatrix}  \]
 are well-defined elements of $\cR$ (although we cannot multiply two such matrices). Instead, it does not make sense to compute the determinant of $r\sigma_2$ as $\cR$ is just an abelian group (written additively) and the product $r (- r)$ is not defined. With the next proposition we investigate further the hyperbolic planes.

\begin{prop}
Consider $\cR = \Z$, $\Z^2_1 = (\Z^2, \sigma_2)$ and $\Z^2_2 = (\Z^2, 2 \sigma_2)$. Then $\Z^2_1$ and $\Z^2_2$ are not isomorphic as $\cR$-symplectic groups.
\end{prop}
\begin{proof}
In order to prove our claim, we show that it does not exist a group homomorphism $\phi: \Z^2_1\to\Z^2_2$ which is also a morphism of $\Z$-symplectic groups.
This can be understood via the following argument: Since any group homomorphism is defined by its action on the generators, it is enough to notice that the elements $(1, 0), (0, 1) \in \Z^2_1$ cannot be mapped to any element of $\Z^2_2$ while preserving the values of the symplectic form. This because $2 \sigma_2$ never takes the value $1$ on $\Z^2_2$. 
 Therefore, $\Hom(\Z^2_1, \Z^2_2) = \void$. 
\end{proof}

\begin{rmk}
The same kind of reasoning apply also for $\Z^2_{r_1}$ and $\Z^2_{r_2}$ for any $r_1, r_2 \in \cR$ that generate different subgroups of $\cR$.
\end{rmk}

Having introduced the notation for hyperbolic planes $\Z^2_r$ we can see that, in favourable conditions, general symplectic abelian groups can be written in terms of them. Before stating our results, we need a preparatory definition.

\begin{defn} \label{defn:finite_abelian group}
We say that an abelian group $\cR$ is of \emph{rank $1$} if every finitely generated sub-group of $\cR$ is cyclic, \ie it is generated by one element. 
\end{defn}

\begin{exa} 
Abelian groups of rank 1 have been completely classified:
\begin{itemize}
\item[-] A torsion-free abelian group of rank 1 is either $\Q$ or a sub-group of $\Q$. 
\item[-] A torsion abelian group of rank $1$ is either $\Q/\Z$ or a sub-group of  $\Q/\Z$.
\end{itemize}
\end{exa}

\begin{thm} \label{thm:symplectic_diagonalization}
Let $(\Z^{2n}, \sigma)$ be a $\cR$-symplectic abelian group and suppose $\cR$ to be of rank $1$. Then 
\begin{equation}\label{eq: diagonalization}
 (\Z^{2 n}, \sigma) \cong (\Z^2_{r_1} \oplus \cdots \oplus \Z^2_{r_n}) 
\end{equation}
for some $r_1, \ldots, r_n \in \cR$.
\end{thm}
\begin{proof}
First of all, we notice that assigning a symmetric bilinear form $\sigma: \Z^{2 n} \times \Z^{2 n} \to \cR$ is equivalent to specify a $2 n \times 2 n$ matrix with coefficients in $\cR$, in a similar fashion of what we explained so far for the case $n = 1$ in Example \ref{rmk:rank_hyperbolic_plane}. Indeed, since a linear morphism $\Z \to \cR$ is uniquely determined by the value of $1$, we have 
\begin{align*}
\Hom_\Z(\Z^{2 n} \otimes_\Z \Z^{2 n}, \cR) &\cong \Hom_\Z(\Z^{4 n^2}, \cR) \cong \Hom_\Z(\Z, \cR)^{4 n^2} \cong \cR^{4 n^2}
\end{align*}
where we used the fact that $\Hom_\Z(\Z, \cR) \cong \cR$. Now consider the sub-group
\[ \cS = \sigma(\Z^{2 n} \otimes_\Z \Z^{2 n}) \subset \cR. \]
Since $\cS$ is a finitely generated abelian sub-group of $\cR$ it is either (abstractly) isomorphic to $\Z$ or to $\Z/m\Z$ for some $m \in \N$. Notice that we can consider $\sigma$ to belong in $\Hom_\Z(\Z^{2 n} \otimes_\Z \Z^{2 n}, \cS)$ and hence we can suppose that $\sigma$ has values on a ring. Since $\Z$ and $\Z/m\Z$ are principal ideal rings we can apply Theorem IV.1 of \cite{NM} to conclude that $(\Z^{2 n}, \sigma)$ can be written in the claimed form. 
\end{proof}

\begin{cor} \label{cor:symplectic_diagonalization}
Theorem \ref{thm:symplectic_diagonalization} remains true if we drop the hypothesis that $\sigma$ is non-degenerate.
\end{cor}
\begin{proof}
  Theorem IV.1 of \cite{NM} is true also for degenerated symplectic forms for which it gives $(\Z^{2 n}, \sigma) \cong (\Z^2_{r_1} \oplus \cdots \oplus \Z^2_{r_{m}} \oplus \cH)$, with $m \le n$ and $\cH$ a subspace where the symplectic form is identically null.
\end{proof}

\begin{notation}\label{not: diagonalization}
With an abuse of language, we say that the isomorphism~\eqref{eq: diagonalization} is a \emph{`diagonalization'} of $\sigma$.
\end{notation}

For the next theorem we need to introduce some notation, based on Theorem \ref{thm:symplectic_diagonalization}. 
\begin{notation}
Let $\sigma: \Z^{2n} \otimes \Z^{2n} \to \cR$ be a symplectic form valued on a rank $1$ group. By identifying $\cR$ with a sub-group of $\Q/\Z$ (or with a sub-group of $\Q$ in the case $\cR$ is torsion-free), we can assume without loss of generality that the symplectic form  $\sigma$ can be  represented by a matrix of the form
\[ \bM(\sigma) = q_\sigma (\bm_{i,j})_{1 \le i,j \le 2n} \]
where $q_\sigma \in \Q/\Z$ (or $\Q$) and $\bm_{i,j} \in \Z$.
\end{notation}

The last piece of notation we need is given by the \emph{external direct sum} of a $\cR_1$-pre-symplectic abelian group with a $\cR_2$-pre-symplectic one. Namely, if $\cG$ is an abelian group and $\sigma_1$ is a $\cR_1$-pre-symplectic form on $\cG$ and $\sigma_2$ is a $\cR_2$-pre-symplectic form (always defined on $\cG$), then $\cG$ equipped with the pre-symplectic form
\[ (\sigma_1 \boxplus \sigma_2)(x, y) = (\sigma_1(x, y), \sigma_2(x, y)) \] 
is a $\cR_1 \oplus \cR_2$-pre-symplectic abelian group. We denote the group so obtained by $(\cG, \sigma_1)\boxplus (\cG, \sigma_2)$.

\begin{thm} \label{thm:symplectic_diagonalization_2}
Let $(\Z^{2n}, \sigma)$ be a $\cR$-symplectic abelian group and assume that $\cR \cong \cR_1 \oplus \ldots \oplus \cR_m$ with $\cR_i$ of rank $1$. Let us rewrite the symplectic form as $\sigma = \sigma_1 \boxplus \ldots \boxplus \sigma_m$  where
$ \sigma_i: (\Z^{2n}, \sigma) \to \cR_i $
and assume that 
$$\bM(\sigma_i)  \bM(\sigma_j)  =  \bM(\sigma_j) \bM(\sigma_i)$$ for all $i,j$, where $    \bM(\sigma_i)$ denotes the matrix associated to $\sigma_i$ as explained above. Then 
\[ (\Z^{2 n}, \sigma) \cong (\Z^{2n}, \sigma_1) \boxplus \cdots \boxplus (\Z^{2n}, \sigma_m) \]
where each $(\Z^{2n}, \sigma_i)$ is diagonal as in Theorem \ref{thm:symplectic_diagonalization}.
\end{thm}
\begin{proof}
It is easy to see that the symplectic form $\sigma: \Z^{2n} \otimes \Z^{2n} \to \cR \cong \cR_1 \oplus \ldots \oplus \cR_m$ can always be written in the form $\sigma = \sigma_1 \boxplus \ldots \boxplus \sigma_m$. Notice that only $\sigma$ is supposed to be non-degenerate, and this does not imply that each $\sigma_i$ is non-degenerate. 
On account of Corollary \ref{cor:symplectic_diagonalization} we can assign to each $\sigma_i$ a matrix of the form
\[ \bM(\sigma_i) = q_{\sigma_i} (\bm_{j,l}). \]
Without any loss of generality, we can set $q_{\sigma_i} = \frac{1}{t_i}$ with $t_i \in \N$. In this way, if we write $t = \prod_{i = 1}^n t_i$ the matrices
\[ \tilde{\bM}(\sigma_i) = \frac{1}{t^{2n}} \left ( \left ( \frac{t}{t_i}\right) \bm_{j,l} \right ) \]
are all integer valued (up to the factor $\frac{1}{t^{2n}}$), and they represent the same symplectic form associated to $\bM(\sigma_i)$, \ie $\sigma_i$. But as all coefficients of $\tilde{\bM}(\sigma_i)$ lie in $\Z[\, \frac{1}{t^{2n}}] \subset \Q$, we can simultaneously identify this subset of $\Q$ with $\Z$ by multiplication by $t^{2n}$ in all the copies of $\Q$ we have chosen. Once these identifications are done, we can assume that $\sigma$ is a symplectic form with values in $\Z^m$, whose diagonalization~
 is equivalent to the existence of a base which diagonalizes all the matrices $\tilde{\bM}(\sigma_i)$ at the same time, that is equivalent to $\bM(\sigma_i)     \bM(\sigma_j) =     \bM(\sigma_j) \bM(\sigma_i)$ for all $i,j$.
\end{proof}

\begin{cor} \label{cor:symplectic_diagonalization_2}
Theorem \ref{thm:symplectic_diagonalization_2} remains true if we drop the hypothesis that $\sigma$ is non-degenerate.
\end{cor}
\begin{proof}
In the proof of Theorem \ref{thm:symplectic_diagonalization_2} we used Corollary \ref{cor:symplectic_diagonalization} that does not require the symplectic forms to be non-degenerate.
\end{proof}

\begin{exa}\label{ex:noncomm sympl form}
Theorem \ref{thm:symplectic_diagonalization_2} cannot be generalized as it is easy to find skew-symmetric matrices which are not simultaneously diagonalizable by just considering two skew-symmetric matrices whose commutant is not zero. Such an example could be
\[
\sigma_1 = \begin{pmatrix}
0 & 1 & 1 & 0 \\
-1 & 0 & 0 & 0 \\
-1 & 0 & 0 & 1 \\
0 & 0 & -1 & 0 \\
\end{pmatrix}, \ \
\sigma_2 = \begin{pmatrix}
0 & 1 & 0 & 0 \\
-1 & 0 & 0 & 0 \\
0 & 0 & 0 & 1 \\
0 & 0 & -1 & 0 \\
\end{pmatrix},
\]
for which it holds
\[ \sigma_1\sigma_2-\sigma_2\sigma_1 = \begin{pmatrix}
-1 & 0 & 0 & 1 \\
0 & -1 & 0 & 0 \\
0 & -1 & -1 & 0 \\
0 & 0 & 0 & -1 \\
\end{pmatrix}- \begin{pmatrix}
-1 & 0 & 0 & 0 \\
0 & -1 & -1 & 0 \\
0 & 0 & -1 & 0 \\
1 & 0 & 0 & -1 \\
\end{pmatrix}= \begin{pmatrix}
0 & 0 & 0 & 1 \\
0 & 0 & 1 & 0 \\
0 & -1 & 0 & 0 \\
-1 & 0 & 0 & 0 \\
\end{pmatrix}.
 \]
\end{exa}

For a generic abelian group $\cG$ there is a lot of freedom in defining symplectic forms on it, but a constrain must be taken into account, as we will see in the next lemma: If $\cG$ contains a torsion sub-group $\cG_\tors$, then the symplectic form cannot take values on a torsion-free group (\ie the bilinear form must be degenerate).
 
\begin{lemma} \label{lem:torsion}
Let $\cR_\free$ be a torsion-free abelian group. Then, for any abelian group $\cG$ such that $\cG_\tors \ne 0$, there does not exist a symplectic bilinear form $\sigma:\cG\times  \cG \to \cR_\free$.
\end{lemma}
\begin{proof}
Suppose by contradiction that it is possible to define a symplectic form on $\cG$ with values in $\cR_\free$
\[ \sigma: \cG \times \cG \to \cR_\free. \]
Then, given any $x \in \cG_\tors$, $y \in \cG$ there exists $n \in \Z - \{0\}$ such that $n x = 0$, hence by bilinearity
\[ n \sigma(x, y) = \sigma(n x, y) = \sigma(0, y) = 0 \then \sigma(x, y) = 0.\qedhere  \]
\end{proof}

\subsection{The symplectic group and its orbits} \label{sec:orbit}

Studying the orbits of the action of $\Sp(\cG,\sigma)$ on $\cG$ will be important for the classification of $\Sp(\cG)$-invariant states, as we shall see in Section~\ref{sec:invariant states}. 
Let us begin by giving the definition of the symplectic group in our context.

\begin{defn} \label{defn:symplectic_group}
Let $(\cG, \sigma)$ be a $\cR$-(pre-)symplectic abelian group. The \emph{symplectic group} of $(\cG, \sigma)$ is defined as
\[ \Sp(\cG, \sigma) \doteq \Sp(\cG) \doteq \{ M \in \Aut(\cG) |  \sigma \circ (M \otimes M) = \sigma \} \subset \Aut(\cG), \]
where $\Aut(\cG)$ is the group of automoprhisms of $\cG$ as an abelian group.
\end{defn}

The definition just given is equivalent to say that $\Sp(\cG)$ is the group of automorphisms of $(\cG, \sigma)$ as an element of $\bP\bSymp_{\cR}$. Hence, it is obvious that $\Sp(\cG)$ is a sub-group of $\Aut(\cG)$.

\begin{exa}
The most classical example of symplectic group is the one associated to $\K^{2n}$ equipped with its canonical symplectic form, \ie $\Sp(\K^{2n},\sigma_{2n})=\Sp(2n, \K)$. Notice that $\Sp(2n, \K)$ is a {sub-group} of the special linear group $\Sl(2n, \K)$,  and for $n = 1$, it is precisely equal to $\Sl(2,\K)$. Moreover, for $\F_2$ we notice that $\Sp(2n,\F_2)=\SO(2n,\F_2)$. 
More generally, these symplectic groups have been thoroughly studied for finite-dimensional vector spaces over any $\K$, see e.g. \cite{A,MT,M,T}. 
\end{exa}

\begin{rmk} \label{rmk:inversion}
    Notice that the symplectic group is never trivial because the automorphism $\Inv: g \mapsto -g$ preserves symplectic form. Indeed, by the linearity of $\sigma$ one has that
    \[ \sigma(-x, -y) = -(-\sigma(x,y))) = \sigma(x,y), \ \ \forall x,y \in \cG. \]
\end{rmk}

We are now ready for describing the orbits of $\Sp(\cG)$ for some important special cases. We introduce the following notion.

\begin{defn} \label{defn:primitive_element}
An element in $x = (x_i) \in \Z^n$ is called \emph{primitive} if $\gcd(x_i) = 1$. 
\end{defn}

Clearly any element $x = (x_i) \in \Z^n$ can be written as $\gcd(x_i) y$, with $y$ primitive.

\begin{prop}\label{lem:Orbite_su_Z}
Suppose that $r \in \cR$ is not a torsion element. Then, the classical symplectic group $\Sp(2 n,\Z)$ is the symplectic group of $(\Z^{2 n}, r \sigma_{2 n})$ and a set of representatives for the orbits of the action of $\Sp(2 n,\Z)$ on $\Z^{2 n}$ is given by $\cE = \{j y \, | \, j \in \N\}$ where $y$ is a fixed primitive element.
\end{prop}
\begin{proof}
The first assertion is easy to check and a proof the fact that the action of $\Sp(2 n,\Z)$ is transitive on primitive elements can be found in Example 5.1 (ii) of \cite{MS}. The characterization of the orbits directly follows form this fact because elements of $\Sp(2 n,\Z)$ have determinant $1$. 
\end{proof}

If $r \in \cR$ is a torsion element, then $(\Z^{2 n}, r \sigma_{2 n})$ can have more automorphisms than the ones coming from $\Sp(2 n,\Z)$, \eg for $r=0$ (the trivial symplectic bilinear form).

\begin{prop}\label{lem:Orbite_su_Z_razionali}
If $r \in \cR$ is a torsion element and $\cG = (\Z^{2 n}, r \sigma_{2 n})$, then 
\[ \Sp(2 n,\Z) \subset \Sp(\cG) \subset \Sl(2 n,\Z). \]
\end{prop}
\begin{proof}
The inclusion $\Sp(\cG) \subset \Sl(2 n,\Z)$ is obvious. The inclusion $\Sp(2 n,\Z) \subset \Sp(\cG)$ follows by standard computations.
\end{proof}

For finite groups the situation is different and we show that for $\F_p^2$ there are only two orbits: one fixed point and its complement. Even if it is a standard result, we would like to recall it.

\begin{lemma}\label{lem:Orbita_su_F_p}
The classical symplectic group $\Sp(2, \F_q)$ is the symplectic group of $(\F_q^2, r \sigma_2)$ for any non torsion $r \in \cR$, any abelian group $\cR$ and $q = p^f$ for prime number $p$. Moreover, for any elements $m,n \in \F^2_q$, different from $(0,0)$,
there exists a $\Theta \in \Sp((\F^2_q, \sigma))$ such that $\Theta n = m$.
\end{lemma}
\begin{proof}
The first claim easily follows from Proposition \ref{lem:Orbite_su_Z}. For the second claim let $m,n\in\F^2_q$ of the form $m=(m_1,m_2)$ and $n=(n_1,n_2)$ and consider a generic ${\Theta} \in \Sp(\F^2_q)=\Sl(\F^2_q)$  satisfying
    \[
    m=\Theta n = \begin{pmatrix}  a n_1 + b n_2 \\ c n_1 +   d n_2 \end{pmatrix}\,.
    \]
Since $\F_q$ is a field, and assuming that both $m$ and $n$ are not the null vector, we can always find $a,b,c,d \in \F_q$ which solve the equations
\begin{align*}
& a n_1 + b n_2= m_1 \qquad  \qquad c n_1 + d n_2= m_2 \qquad  \qquad a d - b c= 1 \,.
\end{align*}
 Indeed, if $n_1 \neq 0 \neq m_1$, we can set  $b=0$, $a=m_1 n_1^{-1}$, $d=n_1m_1^{-1}$ and $c=(m_2-n_2n_1m_1^{-1})n_1^{-1}$;  if $n_1 \neq 0 \neq m_2$ we can set $d=0$, $c=m_2 n_1^{-1}$, $b=n_1m_2^{-1}$ and $a=(m_1-n_2n_1m_2^{-1})n_1^{-1}$; if $n_2 \neq 0 \neq m_1$ we can set  $a=0$, $b=m_1 n_2^{-1}$, $c=n_2m_1^{-1}$ and if $d=(m_2-n_1n_2m_1^{-1})n_2^{-1}$;
and finally for $n_2 \neq 0 \neq m_2$, we can use $c=0$, $d=m_2 n_2^{-1}$, $a=n_2m_2^{-1}$ and $b=(m_1-n_1n_2m_2^{-1})n_2^{-1}$. This concludes our proof.
\end{proof}

Lemma \ref{lem:Orbita_su_F_p} can be readily generalized to $(\F_q^{2 n}, r \sigma_{2n})$ but we postpone the proof of our claim to Lemma \ref{lemma:infinite_torsion_orbits} where we prove a more general result. Also, the description of the orbits of $((\Z/N\Z)^{2 n}, r \sigma_{2n})$ can be easily reduced to the study of congruences as the ones of Lemma \ref{lem:Orbita_su_F_p}.

The description of the symplectic group of a generic $\cR$-symplectic abelian group is much more complicated. A classification of all possible groups that arise in this way seems infeasible, therefore we just discuss some examples that go beyond the classical cases described so far. In order to do that, we introduce some notation. Let
\[ \cG_{n_1, n_2, r} = (\Z^{2 n_1}, r \sigma_{n_1}) \oplus (\Z^{2 n_2}, \sigma_0) \]
where $\sigma_0$ is the trivial pre-symplectic form. Notice that 
\[ \Sp(\cG_{n_1, n_2, r}) \cong \Sp((\Z^{2 n_1}, r \sigma_{n_1})) \times \Sl(\Z^{2 n_2}). \]

\begin{prop} \label{prop:torsionfree_symplectic_group}
    Let $\cG = (\Z^{2n}, \sigma)$ be a $\T$-symplectic abelian group. Then, $\Sp(\cG)$ can be written as an intersection inside $\GL(2 n, \Z)$ of a finite number of conjugates of the symplectic groups of the form $\Sp(\cG_{n_{i, 1}, n_{i, 2}, r_i})$ for some $n_{i, 1}, n_{i, 2} \in \N$ and $r_i \in \T$. 
\end{prop}
\begin{proof}
Since $\Z^{2n}$ has finite rank we can suppose that $\sigma$ has values on a finite rank sub-group of $\T$, that we can write as $\cR_1 \oplus \cdots \oplus \cR_m$, where each factor has rank one. Therefore, we can write $\sigma$ as an external direct sum of $\sigma_1 \boxplus \cdots \boxplus \sigma_m$ where each of $\sigma_k$ has values on a rank one abelian group. By Theorem \ref{thm:symplectic_diagonalization} each $\sigma_k$ is diagonalizable (although not necessarily non-degenerate) and hence $(\Z^{2n}, \sigma_k) \cong \cG_{n_{k, 1}, n_{k, 2}, r_k}$ for some $n_{i, 1}, n_{i, 2} \in \N$ and $r_i \in \T$.

Now, suppose that we have chosen a basis of $\Z^{2n}$ such that $\sigma_1$ is diagonalizable. By definition, we can identify $\Sp((\Z^{2n}, \sigma_1))$ with some $\Sp(\cG_{n_{1, 1}, n_{1, 2}, r_1})$. Then, in general $\sigma_2$ is not diagonalizable in the chosen base but it becomes by a suitable change of base, \ie $\sigma_2=M^t \rho M$ for some diagonal $\rho$. Consider $A\in\Sp(\Z^{2n},\sigma_2)$. Then  we have
$${ M^t \rho M = \sigma_2 = A^t \sigma_2 A = A^t M^t \rho M A \,,}$$
which implies $M A M^{-1} \in \Sp(\Z^{2n}, \rho)$ and therefore isomorphisms
$$\Sp(\Z^{2n},\sigma_2) \cong M^{} \Sp(\Z^{2n}, \rho) M^{-1} \cong \Sp(\cG_{n_{2, 1}, n_{2, 2}, r_2})\,.$$
 Clearly an element of $\GL(2 n, \Z)$ belongs to $\Sp(\Z^{2n}, \sigma_1 \boxplus \sigma_2)$ if and only if it belongs to $\Sp(\Z^{2n}, \sigma_1) \cap M^{} \Sp(\Z^{2n}, \rho) M^{-1}$. Iterating this reasoning finitely many times we get the claimed description of $\Sp(\Z^{2n}, \sigma_1 \boxplus \cdots \boxplus \sigma_m) \cong \Sp(\cG)$.
\end{proof}

\begin{exa} \label{ex:conjugate sympl group}
\begin{enumerate}
\item Let $\cG=(\Z^4,\sigma_1 \boxplus \sigma_2)$ with $\sigma_1$ and $\sigma_2$ as in Example~\ref{ex:noncomm sympl form}. Then $\sigma_1=M^T \sigma_2 M$ with $M$ given by
$$ M= \begin{pmatrix}
1 & 0 & 0 & 0\\
0 & 1 & 0 & 0\\
0 & 0 & 1 & 0\\
1 & 0 & 0 & 1
\end{pmatrix} \,.$$
Then, by Proposition \ref{prop:torsionfree_symplectic_group}, we have
$$ \Sp(\cG) = \Sp(\Z^4,\sigma_1)\cap M \Sp(\Z^4,\sigma_1) M^{-1} = \{ A\in \Sp(4,\Z)\, | \, M A M^{-1} \in \Sp(4,\Z) \}\,.$$
We show that $\Sp(\cG)$ is a ``big" group. For simplifying computations, and without changing the outcome, up to isomorphism, we suppose that
\[ \sigma_1 = \begin{pmatrix}
0 & 0 & 0 & 1\\
0 & 0 & 1 & 0\\
0 & -1 & 0 & 0\\
-1 & 0 & 0 & 0
\end{pmatrix} = \begin{pmatrix}
0 & \Id_2 \\
-\Id_2 & 0
\end{pmatrix}, \ \ \sigma_2 = M^t \sigma_1 M, \]
so that we can use nice blocks representations of matrices. We can also write
\[ M =  \begin{pmatrix}
\Id_2 & 0 \\
A & \Id_2
\end{pmatrix}, \ \ M^t =  \begin{pmatrix}
\Id_2 & A \\
0 & \Id_2
\end{pmatrix}, \ \ M^{-1} =  \begin{pmatrix}
\Id_2 & 0 \\
-A & \Id_2
\end{pmatrix}  \]
where
\[ A = \begin{pmatrix}
0 & 0 \\
1 & 0
\end{pmatrix}. \]
It is well known (see \cite{Stanek}) that $\Sp(\Z^4, \sigma_1)$ is generated by block matrices of the form
\[ T_S = \begin{pmatrix}
\Id_2 & S \\
0 & \Id_2
\end{pmatrix}, \ \ R_U = \begin{pmatrix}
U & 0 \\
1 & (U^t)^{-1}
\end{pmatrix}, \ \ D_Q = \begin{pmatrix}
Q & \Id_2 - Q \\
Q - \Id_2 & Q
\end{pmatrix}
 \]
where $S = S^t$, $\det(U)\ne 0$ and $Q$ is a diagonal matrix with only $0$'s and $1$'s. This generating set is far from being minimal and indeed it is more interesting as a set of elements of $\Sp(\Z^4,\sigma_1)$ than as a generating set. For describing elements of $\Sp(\cG)$ we might check which of the above elements of $\Sp(\Z^4,\sigma_1)$ are also in $M \Sp(\Z^4,\sigma_1) M^{-1}$. By straighforward computations one gets that
\[ M T_S M^{-1} = \begin{pmatrix}
\Id_2 - S A & S \\
-A S A & A S + \Id_2
\end{pmatrix} \]
and that
\[ (M T_S M^{-1}) \sigma_1 (M T_S M^{-1})^t = \]
\[ = \begin{pmatrix}
-S(A + A^t)S & (S A^t)^2 + S A^t + \Id_2 - S A S A^t + S A   \\
-(AS)^2 + A S A^t S - A S + A^t S - \Id_2 & -(AS)^2 A^t - A S A + A(S A^t)^2 + A^t S A^t 
\end{pmatrix}\,. \]
If we write
\[ S = \begin{pmatrix}
s_1 & s_2 \\
s_2 & s_3
\end{pmatrix}, \]
then, in order to have that last expression is equal to $\sigma_1$ it is necessary that
\[ S(A + A^t)S = \begin{pmatrix}
s_1 s_2 + s_2^2 & s_1 s_3 + s_2 s_3 \\
s_2^2 + s_2 s_3 & s_2 s_3 + s_3^2
\end{pmatrix} = 0\,. \]
This implies $s_1 = s_2 = s_3$, together with
\[ -(AS)^2 A^t - A S A + A(S A^t)^2 + A^t S A^t = \begin{pmatrix}
0 & s_3 \\
s_2 & s_1 s_2 + s_1 s_3 
\end{pmatrix} = 0\,. \]
In particular we have $s_2 = s_3 = 0$. Therefore, none of the matrices $T_S$ belong to $\Sp(\cG)$. Similarly 
\[ M R_U M^{-1} = \begin{pmatrix}
U & 0 \\
A U & (U^t)^{-1}
\end{pmatrix} \]
and
\[ (M R_U M^{-1}) \sigma_1 (M R_U M^{-1})^t = \begin{pmatrix}
0 & \Id_2 \\
-\Id_2 & A - U^t A^t (U^t)^{-1}  
\end{pmatrix}.
\]
Therefore, for $R_U \in \Sp(\cG)$, we need to impose the equation
\[ A - U^t A^t (U^t)^{-1} = 0 \]
and if we write
\[ U = \begin{pmatrix}
u_1 & u_2 \\
u_3 & u_4
\end{pmatrix} \]
we get
\[ A - U^t A^t (U^t)^{-1} = \frac{1}{\det U}\begin{pmatrix}
u_1 u_2 & -u_1^2 \\
\det U + u_2^2 & -u_1 u_2
\end{pmatrix}. \]
Hence, we get the conditions
\[ u_1 = 0, \ \ u_2^2 = - \det U = u_2 u_3 \then u_2 = u_3.  \]
Therefore, we get a two dimensional family of elements of $\Sp(\cG)$ parametrized by the matrices of the form $U$ with $u_1 = 0$, $u_2 = u_3 \ne 0$. The same computations for the matrices $D_Q$ are lengthy and we omit them, we just mention that one can show that none of the matrices $D_Q$ belongs to $\Sp(\cG)$. We conclude this example by remarking that the family of elements we described is a two dimensional family that should not describe all the elements of $\Sp(\cG)$ that is expected to be of dimension $5$ as it is described as an intersection inside $\GL(4, \Z)$, that is of dimension $16$, of $\Sp(4,\Z)$ with its conjugate by the matrix $M$, that are of dimension $10$ (here the are considered as schemes over $\Z$). A full description of the group $\Sp(\cG)$ requires an analysis that is too long to be given here.

\item By the description of $\Sp(\Z^{2n}, \sigma)$ given in Proposition \ref{prop:torsionfree_symplectic_group} one expects $\Sp(\Z^{2n}, \sigma)$ to be very small if the rank of the image of $\sigma$ is large enough. Indeed, this already happens if the rank is three because, as we described $\Sp(\Z^{2n}, \sigma)$ as an intersection of closed subvarieties of $\GL(\Z^{2n}, \sigma)$ of codimension $2 n^2 - n$, then one expects to have a minimal generic intersection already when we intersect three of them as $6 n^2 - 3 n > 4 n^2$ if $n > 1$. In such cases $\Sp(\Z^{2n}, \sigma) = \{ \pm \Id \}$. But still, even when the rank of the image of $\sigma$ is large, many interesting examples with  non-trivial symplectic groups are possible for specific choices of matrices of change of bases.

\item Pre-symplectic abelian groups can have a big automorphisms group. For example, if we consider $(\cG, \sigma_0)$, where $\sigma_0$ is the trivial pre-symplectic form, we obtain that $\Sp(\cG)$ is isomorphic to the group of automorphism of the abelian group $\cG$.
\item  $\cG = (\Q^2, \sigma_2)$ is an example of a non-finitely generated $\cR$-symplectic abelian group and $\Sp(\cG) = \Sp(2, \Q)$.
\end{enumerate}
\end{exa}

We refrain to study here the symplectic groups that arise in the case when the underlying abelian group is not finitely generated. The complications involved in the study of such cases go beyond the scope of the present work.

\subsection{$\T$-symplectic abelian groups}

In this section we restrict our attention to the case of $\T = \R/\Z$-valued symplectic forms. 
This is motivated by Lemma \ref{lem:torsion} and by the fact that in many (quantum) physical applications, the usual $\K$-valued symplectic form $\sigma$ is composed with a character of the form $\chi: \K \to S^1 \subset \C$, in order to implement the so-called \textit{canonical commutation relations}, \cf \cite{MSTV}. 

\begin{prop} \label{prop:direct_sum_torsion}
Let $(\cG, \sigma)$ be a $\T$-symplectic abelian group such that $\cG$ is a torsion group and
\[ \cG \cong \cG_1 \oplus \cG_2 \]
with $\cG_1$ and $\cG_2$ of coprime torsion. Then $\sigma$ restricts to a symplectic form both on $\cG_1$ and on $\cG_2$;  moreover $\cG_2 = \cG_1^\perp$ and $\cG_1 = \cG_2^\perp$.
\end{prop}
\begin{proof}
It is enough to show that for any $x \in \cG_1$ and $y \in \cG_2$ then $\sigma(x, y) = 0$ (where we are identifying $\cG_1$ and $\cG_2$ with their image in $\cG$ via the canonical morphisms). Indeed, by the non-degeneracy of $\sigma$ there exists $x' \in \cG$ such that $\sigma(x,x')\ne 0$, therefore this must necessarily be in $\cG_1$. The same thing is true for $y$.

It remains to prove the claim. Suppose that there exist $x \in \cG_1$ and $y \in \cG_2$ such that $\sigma(x, y) \ne 0$. Then, consider $n,m \in \Z$, coprime such that $n x = 0, m y = 0$. Then,
\[ \sigma(x, n y) = n \sigma(x, y) = \sigma(n x, y) = \sigma(0, y) = 0 \]
and on the other hand 
\[ \sigma(m x, y) = m \sigma(x, y) = \sigma(x, m y) = \sigma(x, 0) = 0. \]
But the element $\sigma(x, y)$ cannot be both of $n$ and $m$ torsion because $m$ and $n$ are coprime. Therefore, $\sigma(x,y)=0$.
\end{proof}

\begin{cor} \label{cor:direct_sum_torsion}
Let $(\cG, \sigma)$ be a torsion $\T$-symplectic abelian group, then
\[ (\cG, \sigma) \cong \bigoplus_{p \in \P} (\cG_p, \sigma_p) \]
where $\cG_p$ is the $p$-primary part of $\cG$ and $\sigma_p$ the restriction of $\sigma$ to $\cG_p$.
\end{cor}
\begin{proof}
This is an immediate application of Proposition \ref{prop:direct_sum_torsion} using the fact that any torsion abelian group $\cG$ can be written as
\[ \cG \cong \bigoplus_{p \in \P} \cG_p, \]
where $\cG_p$ is the $p$-primary part.
\end{proof}

\section{Symplectic twisted group algebras}\label{sec:group algebra}
Twisted group algebras arise in a wide variety of situations in mathematics, \textit{e.g.} in the study of representations of nilpotent groups and connected Lie groups~\cite{Moore,Packer,Hanna}, noncommutative differential geometry~\cite{Connes}, the study of continuous $C^*$-trace~\cite{ER,RW}, the algebraic approach to quantum statistical mechanics~\cite{OA1} and quantum field theory~\cite{AQFT1,AQFT2}, and in the realization of octonions, Clifford algebras and multiplicative invariant lattices in $\R^n$~\cite{Albu1,Albu2,Albu3}.  In this paper, we shall focus our attention to twisted group algebras arising from symplectic abelian groups. Since there exists an extensive literature on this topic, see e.g.~~\cite{OA2,OA3}, we just introduce them briefly. \\
Let $(\cG,\sigma)$ be a $\T$-pre-symplectic abelian group. From now on we use the multiplicative notation for the operation of $\cG$. We will denote the identity of $\cG$ with $1_\cG$. It is also convenient to embed $\T$ into $\C^\times$ as the complex number of modulus $1$. This embedding can be realized by the map
$\phi :\R/\Z \to \mathbb \C^\times$ defined by $\phi(x+\Z)=e^{\imath \pi x} \,.$ As a consequence, the symplectic form $\sigma$ is replaced by  $\Omega: \cG \times \cG \to \C^\times$ which is defined by
\[ \Omega(x, y) \doteq e^{\imath \pi \sigma(x,y)} . \]
\begin{prop}
Let $(\cG, \sigma)$ be a $\T$-pre-symplectic abelian group. Then, for any $x,y\in\cG_\tors$, $\Omega(x,y)$ is a root of unity.
\end{prop}
\begin{proof}
As shown in the proof of Lemma \ref{lem:torsion}, if $x,y\in\cG_\tors$ then $\Omega(x,y)$ must be a torsion element of $\C^{\times}$. These are precisely the roots of unity.
\end{proof}
We observe that since $\sigma$ is bilinear and skew-symmetric $\Omega$ defines a group $2$-cocycle. Indeed, for any $1_\cG,g_1,g_2,g_2\in\cG$ we have $\Omega(1_\cG,g_1)=\Omega(g_1,1_\cG)=1$ together with 
\begin{align*}
\Omega(g_1, g_2) \Omega(g_1 g_2, g_3) &= \Omega(g_1, g_2 ) \big(\Omega(g_1, g_3) \Omega(g_2, g_3) \big)= \\
&= \big( \Omega(g_1, g_2 ) \Omega(g_1, g_3) \big) \Omega(g_2, g_3) = \Omega(g_1, g_2 g_3) \Omega(g_2, g_3)\,.
\end{align*}
 Next, we denote by \emph{twisted group algebra $\C[\cG]^\Omega$} the set of all finite $\C$-linear combinations 
\begin{equation*}
\aa=\sum_{g\in\cG} \a_g \,g \qquad \text{with $\a_g\in \C$} \,,
\end{equation*}
endowed with the twisted product defined by
\begin{equation}\label{eq: twisted prod}
 \aa\bb=  \sum_{g\in\cG} \left(\sum_{h\in\cG } \Omega(g,h) \a_h \beta_{h^{-1} g}\right) g.
\end{equation}
It is straightforward to verify that twisting the product by a (non-trivial) group $2$-cocycle makes the algebra $\C[\cG]^\Omega$ noncommutative and associative, although $\cG$ itself is a commutative group. Moreover, the isomorphism class of the twisted algebra so obtained depends only on the group cohomology class of the $2$-cocycle and any cohomology class  in $H^2(\cG, \T)$ can be represented by a pre-symplectic form (\cf \cite[Theorem 7.1]{Kleppner}).
We shall see in Theorem~\ref{thm:CTori} that this twisting reduces considerably the number of $\Sp(\cG)$-invariant states on $\C[\cG]^\Omega$ when the form is non-degenerate.\\
In order to distinguish positive elements, we promote $\C[\cG]^\Omega$ to a twisted group $*$-algebra by endowing it with the involution~$*:\C[\cG]^{\Omega}\to \C[\cG]^{\Omega}$ given by
\begin{equation}\label{eq: invol}
\aa^*:= \sum_{g\in\cG} \overline{\a}_g g^{-1} \,.
\end{equation}
Indeed, any positive element $\bb$ in a complex $*$-algebra $\cA$ can be written as $\bb=\aa^*\aa$ for a suitable $\aa \in \cA$ -- for more detalis see e.g.~\cite{dixmier}.

We are thus ready to summarize this short discussion with the following definition.
\begin{defn}
We denote by $\cW_{\cG, \Omega} \equiv \cW_\cG$ the symplectic twisted group $*$-algebra (\emph{\stg-algebra}) obtained endowing $\C[\cG]^\Omega$ with the product \eqref{eq: twisted prod} and with the involution \eqref{eq: invol}.
\end{defn}

\begin{rmk}
 {When completed by a canonical $C^*$-norm, the twisted group algebras $\cW_{\cG,\Omega}$ are also called Weyl $C^*$-algebras or exponential Weyl algebras in the literature, see e.g. \cite{Moretti,Robinson1,Robinson2}. The reader should not confuse these algebras with the rings of differential operators with polynomial coefficients which are also called Weyl algebras, see e.g.~\cite{Dixmier1,Dixmier2}.}
\end{rmk}

\begin{exa}
The natural settings for \stg-algebras are quantum mechanics and bosonic quantum field theory. Since there are many different bosonic QFT, let us consider for simplicity the case of a scalar field $\phi$ satisfying the wave equation and denote with $Sol$ the space of solutions with spatially compact support (\textit{cf.} \cite{MurroGinoux,MurroGrosse}). Then, we have
$$\cW_{\Omega_{QM}}=\C[C_c^\infty(\R^{6})]^{\Omega_{QM}} \qquad \text{ and } \qquad \cW_{\Omega_{QFT}}=\C[\mathcal{E}(Sol)]^{\Omega_{QFT}}\,, $$
where  $C^\infty_c(\cdot)$ denotes the compactly supported smooth functions, $\mathcal{E}(\cdot)$ is the space of complex linear functionals, the group 2-cocycles are given by $\Omega_{i}(\cdot,\cdot)=e^{-\imath\hbar \sigma_i(\cdot,\cdot)}$,  $\sigma_{QM}$ is obtained using the canonical symplectic form of $\R^6$ as in Example \ref{exa:QFT}, while $\sigma_{QFT}$ by~\eqref{eq:sigma_QFT}.

 Notice that have implicitly assumed that the spacetime is $\R^4$. For a generic globally hyperbolic spacetime, namely a $n$-dimensional oriented Lorentzian manifold diffeormophic to $\R\times \Sigma$, being $\Sigma$ a Cauchy surface, the \stg-algebras are respectively defined as
$$\cW_{\Omega_{QM}}=\C[C_c^\infty(T^*\Sigma)]^{\Omega_{QM}} \qquad \text{ and } \qquad \cW_{\Omega_{QFT}}=\C[\mathcal{E}(Sol)]^{\Omega_{QFT}}\,, $$
being $T^*\Sigma$ the cotangent bundle of $\Sigma$.
For more details, we refer to~\cite{FR16}. It is interesting to notice that, while the representations of $\cW_{QM}$ are all unitary equivalent, according to the Stone-Von Neumann's Theorem~\cite{SVNthm}, there exist plenty of inequivalent representations of $\cW_{QFT}$ according to Haag's Theorem~\cite{Hthm}. Finally, we observe that to each positive element of $\cW$ corresponds a unique physical observable associated to the theory. 
\end{exa}

\begin{exa}\label{ex:NCT}
One of the most studied objects in noncommutative geometry is the so-called \textit{noncommutative torus}. It is defined as the universal $C^*$-algebra~$\fA_\theta$ generated by unitaries $U,V \in C(\T^2,\C)$
satisying the \textit{canonical commutation relations}
\begin{equation}\label{eq:CCR NCT}
UV = e^{2\pi\i\theta} VU \,.
\end{equation}
It is easy to see that $\fA_\theta$ is isomorphic to the $C^*$-completion of the \stg-algebra $\cW_{\Z^2,\Omega}$, 
 where $\Omega$ is the group 2-cocycle on $\Z^2$ given by
$$ \Omega(n,m)=e^{2\imath\pi \theta \sigma_2(n,m)} \qquad \theta\in \R\,,$$
being $\sigma_2$ the canonical symplectic form on $\Z^2$.
Indeed, this twisted group $C^*$-algebra is generated by the unitary operators $U=(1,0)$ and $V=(0,1)$ acting on $\ell^2(\Z^2)$ 
$$ U f(n,m)=e^{-2\pi\imath m} f(n+1,m) \qquad \text{and} \qquad V f(n,m)=e^{-2\pi\imath n} f(n,m+1)\,.$$
By a straightforward computation it is easy to see that this generators satisfy the canonical commutation relation~\eqref{eq:CCR NCT}.
\end{exa}

\begin{rmk}\label{rmk:functor}
Notice that the association $(\cG, \sigma) \mapsto \cW_{\cG, \Omega}$ is a functor from the category $\bP \bSymp_{\T}$ of $\T$-pre-symplectic abelian groups to the category of $*$-algebras over $\C$. Therefore, the action of the symplectic group $\Sp(\cG)$ on $\cG$ can be lifted to an action on $\cW_\cG$ by algebra automorphism via the formula  
\begin{equation*}
\Phi_\Theta(\aa)=\sum_{g\in\cG} \a_g \,\Theta(g) \quad\qquad \text{with $\,\a_g\in \C\,$ and $\,\Theta \in \Sp(\cG)$.}
\end{equation*}
If  the only invariant subspace of $\cW$ under the action of $\Phi_\Theta$, for all $\Theta\in \Sp(\cG)$, is  span$_\C\{1_\cG\}$  then the action of $\Sp(\cG)$ is said \emph{ergodic}.
\end{rmk}
\begin{prop} \label{prop:ergodic}
For a symplectic abelian group $(\cG,\sigma)$ the following statements are equivalent: 
\begin{itemize}
\item[-] The only finite orbit (\ie with finite cardinality) of $\Sp(\cG)$ on $\cG$ is the fixed point $1_\cG$;
\item[-] The action of $\Sp(\cG)$ on $\cW_{\cG}$ is ergodic.
\end{itemize}
\end{prop}
\begin{proof}
Let $g\in\cG$ have finite orbit under $\Sp(\cG)$ and denote it by $\cO \subset \cG$. Then, any element of the form $\aa=\sum_{h\in\cO} h$ is a fixed point for the action of $\Sp(\cG)$ on $\cW_{\cG}$, namely for any $\Theta\in\Sp(\cG)$ we have $\Phi_\Theta(\aa)=\aa$. By logic transposition, this implies that if the action is ergodic then non-trivial finite orbits for the action of $\Sp(\cG)$ on $\cG$ do not exist.\\
Conversely, if the action is not ergodic, then there exists at least an element $\aa\neq \a 1_\cG$, for all $\a\in\C$, such that, for all $\Theta\in\Sp(\cG)$, it satisfies $\Phi_\Theta (\aa)=\aa$.
Now write $\aa$ as the finite sum 
\[ \aa=\sum_{g \in \cO} a_g g \] 
with $\cO$ a finite subset of $\cG$. Since $\Theta$ acts on $\cW_\cG$ via the formula
\[ \Phi_\Theta(\aa)=\sum_{g \in \cO} a_g \Theta(g) = \aa \] 
we have that $\Theta(g) = g'$ for a $g' \in \cO$ and $a_g = a_{g'}$ because the elements $g$ form a $\C$-base of $\cW_{\cG}$. It follows that $\cO$ is a finite union of finite orbits of $\Sp(\cG)$ on $\cG$.
\end{proof}

\section{$\Sp(\cG)$-invariant states}\label{sec:invariant states}

In this section, we keep the notation of the last section and we denote the \stg-algebra associated to a $\T$-pre-symplectic abelian group $(\cG, \sigma)$ simply as $\cW_{\cG}$.
Let now $\omega$ be a \emph{state}, namely a linear functional from $\cW_{\cG} $ into $\C$ that is positive, \ie $\omega(\aa^*\aa)\geq 0$ for any $\aa\in\cW_{\cG} $, and normalized, \ie $\omega(1_\cG)=1$. 

\begin{defn} 
We say that a state $\omega$ is \emph{$\Sp(\cG)$-invariant} if for any $*$-automorphism~$\Phi_\Theta\in \Aut(\cW_{\cG})$, with $\Theta\in \Sp(\cG)$, it holds $\omega \circ \Phi_\Theta =  \omega \,. $
\end{defn}
\begin{rmk}
As anticipated in Example~\ref{ex:NCT}, the \stg-algebra $\cW_{\cG}$ can be equipped with a $C^*$-norm in a canonical way, and hence it can be completed to a $C^*$-algebra $\overline{\cW}_{\cG}$. Since every positive functional on $\cW$ can be extended to a positive functional $\overline{\cW}_{\cG}$ (\cf \cite{Naimark}) and every positive functional on $\overline{\cW}_{\cG}$ can be restricted to a positive functional on $\cW_{\cG}$, then in the rest of the paper we shall focus only on~$\cW_{\cG}$. 
\end{rmk}
In order to construct a ($\Sp(\cG)$-invariant) state $\omega$ on the \stg-algebra $\cW_{\cG}$, it is enough to prescribe its values on each generator of $\cW_\cG$ and then to extend it by linearity to any element $\aa\in\cW_{\cG}$. So, the $\Sp(\cG)$-invariance condition for a state $\omega$ can be written 
\begin{equation}\label{eq:state on W}
\omega(\Theta(g))=\omega(g)=\begin{cases} 1 & \text{ if } g=1_\cG \\ q^{(g)} \in \C & \text{ else} \end{cases}
\end{equation}
for a sequence of values $q^{(g)}$.  Following \cite[Section~2]{BMP}, we can prove the following.

\begin{prop}\label{prop:real_values}
Let $\cW_{\cG} $ be a \stg-algebra and consider a $\Sp(\cG)$-invariant state $\omega$. Then, for any $g\in \cG$, it holds
$$\omega(g) \in [-1,1] \,.$$
\end{prop}
\begin{proof}
Since $\omega$ is a linear positive functional, then, for every $g\in\cG$, it holds
\begin{align*}
\omega\left( (g + 1_\cG)^*(g + 1_\cG)\right) =   2  +\omega\left(g^*\right) + \omega\left(g\right)  \in [0,\infty)  \,.
\end{align*}
According to Remark \ref{rmk:inversion}, there exists an element $\Inv\in\Sp(\cG)$ defined by $g\mapsto g^{-1}$ and it holds
    \[
    \ol{\omega(g)} = \omega(g^*) =\omega(g^{-1})=\omega(\Inv g) = \omega(g)  \,, 
    \]
    where in the fourth equality we used the invariance of the state under the action of the symplectic group. Let now be $\aa_\pm=g \pm 1_\cG$. Since $\omega$ as a positive functional has to satisfy
    $$0\leq \frac{1}{2}\, \omega(\aa_\pm^*\aa_\pm)= 1 \pm \omega(g) \,, $$
     which concludes our proof.
    \end{proof}
    
As anticipated in Section~\ref{sec:group algebra}, the twisting of the product of $\C[\cG]$ by $\Omega$ plays an important role in the characterization of the $\Sp(\cG)$-invariant states. Indeed, if we simply consider the (untwisted) group $*$-algebra $\C[\cG]$ (that corresponds to considering the trivial pre-symplectic form on $\cG$), then there are uncountably many $\Sp(\cG)$-invariant states, as shown in the next proposition.

\begin{prop} \label{thm:CTori}
    Let $\C[\cG]$ be the \emph{(untwisted)} group $*$-algebra associated to the symplectic abelian group $(\cG, \sigma_0)$, where $\sigma_0$ is the trivial pre-symplectic form. Then, there exist infinitely many $\Sp(\cG)$-invariant states.
\end{prop}

\begin{proof}
    It is easy to notice that any constant functional given by
    \begin{equation}\label{eq:state on W_CTori}
    \omega(g)=\begin{cases} 1 & \text{ if } g = 1_\cG \\ q \in [0,1]\bigcap \R & \text{ else} \end{cases}
    \end{equation}
    is $\Sp(\cG)$-invariant (in this case $\Sp(\cG)$ agrees with the group of automorphisms of $\cG$). We need just to verify that it is positive.
    To this end, we notice that for every finite dimensional subspace of $\cV \subset \C[\cG]$, the map $\aa \mapsto \omega(\aa^*\aa)$ is a $\R$-valued quadratic form, therefore it can be written as
    \begin{equation*}
    \omega(\aa^*\aa) = \overline{\a}^t \, \bH \, \a
    \end{equation*}
    where  $\bH$ is a Hermitian matrix and $\a$ a vector with components $\a_j\in\C$. On the $(d+1)$-dimensional subspace spanned by the elements $\{ 1_\cG, h_j \}_{1 \le j \le d}$ the entries of the matrix $\bH$ can be described as 
    \begin{align} \label{eq:H-entries1}
    (\bH)_{0,0}&=(\bH)_{j,j} = 1, \qquad &   j \geq 1\\
    \label{eq:H-entries2}
    (\bH)_{i,j} &= (\bH)_{j,i} = q, \qquad &  i\neq j
    \end{align}
    where \eqref{eq:H-entries1} holds because of the state normalization condition.
    By a straighforward computation, we have that the eigenvalues of $\bH$ are given by
    $$\lambda_1=dq+1 \qquad \text{and} \qquad \lambda_i=1-q \quad \text{ for $i=1,\dots,d$\,.}$$
    Since analogous considerations hold for subspaces of $\C[\cG]$ not containing $1_\cG$ (and hence for any finite dimensional subspace), we can conclude that $\omega$ given by~\eqref{eq:state on W_CTori} is indeed a positive and normalized $\Sp(\cG)$-invariant functional. 
\end{proof}
It is an easy corollary of Proposition~\ref{thm:CTori} that any abelian group equipped with a degenerate $\T$-pre-symplectic form always admit many invariant states (we will discuss a proof of this fact in Section~\ref{sec:conclusion}). In particular, it is immediate to check that the rational noncommutative torus admits many
invariant states as it comes equipped with a degenerate $\T$-pre-symplectic form.

Now we start to study $\Sp(\cG)$-invariant states on $\cW_{\cG}$, with the hypothesis that the product is twisted in a non-trivial way and that the form is non-degenerate.

\subsection{$\cG$ is torsion-free}\label{sec:torsfree}

In this section, we analyze invariant states in the case of torsion-free groups. In particular, in Theorem~\ref{thm:NCtori_no_torsion} we provide a sufficient condition for the symplectic abelian group $(\cG,\sigma)$ to have a unique invariant state. The proof is obtained by reducing to the case solved in our previous work \cite{BMP}. The aim of this section is to deal with the case when $\cG$ is not finitely generated. To this end, we need preparatory lemmas and definitions. Recall that the rank of an abelian group $\cG$ is defined as the dimension of the $\Q$-vector space $\cG \otimes_\Z \Q$.

\begin{defn} \label{defn:completely_decomposable_group}
Let $\cG$ be a torsion-free abelian group, we say that $\cG$ is a \emph{completely decomposable} if $\cG \cong \bigoplus_{i \in I} \cG_i$ where $\cG_i$ is a rank $1$ torsion-free group, \ie a sub-group of $\Q$.
\end{defn}

The class of completely decomposable abelian groups is an amenable class of non-finitely generated abelian groups , including the class of $\Q$-vector spaces.  Nevertheless, groups appearing in applications might be not completely decomposable 

\begin{exa}
The additive group $\Z_p$ of $p$-adic integers is not completely decomposable (it follows from \cite[Example 7.4]{Bae}).
\end{exa}

Given a torsion-free $\T$-symplectic abelian group $(\cG, \sigma)$ we can always consider the $\Q$-vector space $\cG \otimes_\Z \Q$ generated by $\cG$ and extend the symplectic form to $\cG \otimes_\Z \Q$ by considering the $\T$-symplectic form
\[ (\sigma \otimes_\Z \Q)(q_1 g_1, q_2 g_2) = q_1 q_2 \sigma(g_1, g_2),  \]
for any $g_1, g_2 \in \cG$ and $q_1, q_2 \in \Q$.

\begin{lemma} \label{lemma:unique_hyperbolic_plane}
Let $(\cG, \sigma)$ be a $\T$-symplectic abelian group of rank $2$, then $\sigma = (r \sigma_2)|_{\cG}$ for a $r \in \T$.
\end{lemma}
\begin{proof}
By hypothesis $\cG \otimes_\Z \Q \cong \Q^2$. Let $x,y \in \cG$ be two generators of $\cG \otimes_\Z \Q$ as a $\Q$-vector space. Any element in $\cG \otimes_\Z \Q$ can be written as $q_x x + q_y y$ with $q_x, q_y \in \Q$. By writing
\[ q_x = \frac{n_x}{d_x}, \ \ q_y = \frac{n_y}{d_y} \]
we get by bilinearity that
\[ \sigma(q_x x, q_y y) = n_x n_y \sigma \left(\frac{x}{d_x}, \frac{y}{d_x} \right), \]
and the relations
\[ d_x \frac{x}{d_x} = x, \ \ d_y \frac{y}{d_y} = y \]
again by bilinearity imply
\[ \sigma(x, y) = d_x d_y \sigma \left(\frac{x}{d_x}, \frac{y}{d_x} \right) \then \sigma \left(\frac{x}{d_x}, \frac{y}{d_x} \right) = \frac{\sigma(x, y)}{d_x d_y}. \]
This implies that the knowledge of the value $\sigma(x, y)$ uniquely determines the symplectic form $\sigma$ as $\cG$ injectively embeds in $\cG \otimes_\Z \Q$.
\end{proof}

\begin{defn} \label{defn:irrational}
We say that the symplectic form $\sigma: \cG \otimes \cG \to \T$ is \emph{irrational} if 
$$\sigma(\cG \otimes \cG) \cap \Q/\Z = \{1\}\,.$$
\end{defn}

The name \emph{irrational} used in Definition \ref{defn:irrational} comes from the case when $\cG = \Z^{2 n}$ in which case the corresponding $\sigma$ is the symplectic form whose associated \stg-algebra is that of an irrational noncommutative torus (see Example \ref{ex:NCT}). Namely, there is an algebraic isomorphism of group
\[ \T = \Q/\Z \oplus \bigoplus_{\theta} \Q \theta \]
where $\theta$ runs over a $\Q$-base of irrational numbers modulo $1$. The irrational numbers $\theta$ are precisely the angles of irrational noncommutative tori, whereas if $\sigma$ has values in the subgroup $\Q/\Z$ of roots of unity the associated \stg-algebra is the one of a rational noncommutative torus. Another way to state Definition \ref{defn:irrational} is that it asks for a symplectic form whose values do not lie in the torsion subgroup of $\T$.

We say that the state on the \stg-algebra $\cW_\cG$ defined by the values
$$
\tau(g)=\begin{cases} 1 & \text{ if } g=1_\cG \\ 0 & \text{ otherwise } \end{cases}
$$
and extended to $\cW_\cG$ by linearity is the \emph{tracial state}. We also often use the following hypothesis.

\begin{ass} \label{ass:diagonalization}
On any finitely generated sub-group $\cF \subset \cG$ the restriction of $\sigma$ to $\cF$ satisfies the hypothesis of Theorem \ref{thm:symplectic_diagonalization_2}.
\end{ass}

A first result about non-finitely generated torsion-free groups is the following.

\begin{lemma} \label{lemma:NCtori_no_Q2}
Let $\cG$ be a $\T$-symplectic abelian group of rank $2$ such that 
$\sigma$ is irrational. Then, the only $\Sp(\cG)$-invariant state on $\cW_\cG$ is the tracial state.
\end{lemma}
\begin{proof}
By Lemma \ref{lemma:unique_hyperbolic_plane} the claim follows by a direct application of Theorem 2.2 of \cite{BMP}.
\end{proof}

The next step is to generalize the previous lemma to all injective abelian groups.

\begin{lemma} \label{lemma:NCtori_no_torsion_injective}
Let $(\cG, \sigma)$ be an injective $\T$-symplectic abelian group such that 
$\sigma$ is irrational and that satisfies Assumption \ref{ass:diagonalization}. 
Then, the only $\Sp(\cG)$-invariant state on $\cW_\cG$ is the tracial state.
\end{lemma}
\begin{proof}
The condition of $\cG$ being injective is equivalent to $\cG$ being a $\Q$-vector space. By Proposition \ref{prop:hyperbolic plane} we know that there is a hyperbolic plane $\cH \subset \cG$ that contains $g$. Since $\cG$ is a $\Q$-vector space then we can consider the $\Q$-vector space generated by $\cH$ in $\cG$, that we denote $\cH_\Q$. The restriction of $\sigma$ to $\cH_\Q$ is a $\T$-symplectic form, hence $\cH_\Q$ is isomorphic to $(\cH_\Q, r \sigma_2)$, for some irrational $r$, as a consequence of Lemma \ref{lemma:unique_hyperbolic_plane}. Since $\sigma$ is diagonalizable, it is easy to see that 
\[ (\cG, \sigma) \cong (\cH_\Q, \sigma|_{\cH_\Q}) \oplus (\cH_\Q^\perp, \sigma|_{{\cH_\Q}^\perp}) \]
as a $\T$-symplectic abelian group, because it is easy to check that $\sigma(h, h') = 0$ for $h \in \cH_\Q$, $h' \in \cH_\Q^\perp$. Therefore, we get an inclusion $\Sp(\cH_\Q) \subset \Sp(\cG)$. This implies that every $\Sp(\cG)$-invariant state $\omega$ on $\cW_\cG$ must restrict to a $\Sp(\cH_\Q)$-invariant state on $\cW_{\cH_\Q}$. This implies that $\omega$ is trivial on $\cH_\Q$ by Lemma \ref{lemma:NCtori_no_Q2}.
\end{proof}

The next one is our first general result.

\begin{thm} \label{thm:NCtori_no_torsion_CD}
Let $(\cG, \sigma)$ be completely decomposable $\T$-symplectic abelian group such that 
$\sigma$ is irrational and that satisfies Assumption \ref{ass:diagonalization}. 
Then, the only $\Sp(\cG)$-invariant state on $\cW_\cG$ is the tracial state.
\end{thm}
\begin{proof}
We can embed $\cG$ in its injective envelop $\cG \otimes_\Z \Q$, that comes equipped with a structure of $\T$-sympletic abelian group $\sigma \otimes_\Z \Q$ that canonically extends the one of $\cG$, and for any $g \in \cG$ we can consider a $\Q$-hyperbolic plane $\cH_\Q \subset \cG \otimes_\Z \Q$ containing $g$. Now, suppose that there exists a non-trivial $\Sp(\cG)$-invariant state on $\cW_\cG$, then this state would be such that $\omega(g) \ne 0$ for some $g \ne 1_\cG$. For such a $g$ we have that $(\cG, \sigma) \cong (\cH_\Q \cap \cG, \sigma) \oplus ((\cH_\Q \cap \cG)^\perp, \sigma)$, by the same reasoning used in Lemma \ref{lemma:NCtori_no_torsion_injective}. Therefore, $\Sp((\cH_\Q \cap \cG, \sigma))$ embeds in $\Sp(\cG)$, proving that any $\Sp(\cG)$-invariant state on $\cW_\cG$ must restrict to the trivial state on $(\cH_\Q \cap \cG, \sigma)$, contradicting to the hypothesis that $\omega(g) \ne 0$.
\end{proof}

In the next theorem we axiomatize the proof of Theorem \ref{thm:NCtori_no_torsion_CD}. In order to do that we introduce the following concepts. 

\begin{defn} \label{defn:split_hyperbolic_plane}
Let $(\cG, \sigma)$ be a $\T$-symplectic abelian group. Let $\cH \subset (\cG, \sigma)$ be a rank $2$ hyperbolic plane (by this we mean a sub-group of rank $2$ of $\cG$ where $\sigma$ restricts to a non-degenerate form). We say that $\cH$ is \emph{split} if
\[ \cG \cong (\cH, \sigma_\cH) \oplus (\cH^\perp, \sigma_{\cH^\perp}) \]
as a $\T$-symplectic abelian group.
\end{defn} 

Notice that for any split hyperbolic plane $\cH \subset (\cG, \sigma)$ we have an injective morphism
\[ \Sp(\cH) \to \Sp(\cG). \]

\begin{defn} \label{defn:plane_automorphisms}
Let $(\cG, \sigma)$ be a $\T$-symplectic abelian group. Let $\{ \cH_i \}_{i \in I}$ be the family of split hyperbolic planes of $\cG$. We define the group of \emph{plane automorphisms} of $\cG$ as 
\[ \lt \{ \Sp(\cH_i) \}_{i \in I} \gt \subset \Sp(\cG), \]
and we denote it by $\Sp^H(\cG)$.
\end{defn}

\begin{thm} \label{thm:NCtori_no_torsion}
Let $(\cG, \sigma)$ be a $\T$-symplectic abelian group such that 
$\sigma$ is irrational and satisfies Assumption \ref{ass:diagonalization}. Assume also that $\Sp(\cG)\otimes_\Z \Q \supset \Sp^H(\cG \otimes_\Z \Q)$, where $\cG \otimes_\Z \Q$ is equipped with the  $\T$-symplectic form $\sigma \otimes_\Z \Q$.
Then, the only $\Sp(\cG)$-invariant state on $\cW_\cG$ is the tracial state.
\end{thm}
\begin{proof}
Let $\cG \otimes_\Z \Q$ be the injective envelope of $\cG$. Notice that every automorphism of $\cG$ lifts to an automorphism of $\cG \otimes_\Z \Q$ because of the universal property of $\cG \otimes_\Z \Q$. 
This gives an injection of $\Sp(\cG)$ in $\Sp(\cG \otimes_\Z \Q)$ because $\cG \otimes_\Z \Q$ is equipped with the $\T$-symplectic form $\sigma \otimes_\Z \Q$. Since we assume that ${\Sp(\cG)\otimes_\Z \Q} \supset \Sp^H(\cG \otimes_\Z \Q)$, this implies that every $\Sp(\cG)$-invariant state induces an invariant state on $\cW_{\cG \cap \cH_i}$, for any split hyperbolic plane $\cH_i$ of $\cG \otimes_\Z \Q$. And this restriction must be trivial as a consequence of Lemma \ref{lemma:NCtori_no_Q2}. Since each non-null element of $\cG \otimes_\Z \Q$ generates a split hyperbolic plane, we get that only the only $\Sp(\cG)$-invariant state on $\cW_\cG$ is the trivial state.
\end{proof}

\begin{rmk}
Without the hypothesis of diagonalization of Theorem~\ref{thm:symplectic_diagonalization_2}, the symplectic group could be very small, as discussed in Example~\ref{ex:conjugate sympl group}. Therefore, its action could not be ergodic, allowing many invariant states in some specific cases.
\end{rmk}

\subsection{$\cG$ is torsion and non-finitely generated}\label{inftors}

In this section  we investigate the uniqueness of $\Sp(\cG)$-invariant states for torsion $\T$-sympletic abelian groups $\cG$ of infinite rank. Since a complete classification of this class of groups is not known, we shall focus our attention on some important examples. 
First, consider an infinite direct sum
\[ \cG = \bigoplus_{i \in I} (\F_p^2, \sigma_2) \]
where $I$ is any set of infinite cardinality (as if $I$ is finite than the action of $\Sp(\cG)$ on $\cW_\cG$ is not ergodic, on account of Proposition~\ref{prop:ergodic}, and it is not difficult to see that in this case $\cW_\cG$ has many invariant states).  
Without loss of generality, we assume $I = \N$, since the cases when $I$ has bigger cardinality can be dealt in the same way or reduced to this case.

\begin{lemma} \label{lemma:infinite_torsion_orbits}
The action of $\Sp(\cG)$ on $\cG = \bigoplus_{i \in I} (\F_p^2, \sigma_2)$ has only one orbit besides the one of the identity of $\cG$.
\end{lemma}
\begin{proof}
Let $g, h \in \cG$ be two non null-elements. By definition, we can write $g = (g_i)$, $h = (h_i)$ with $g_i,h_i \in \F_p^2$ where only a finite number of $g_i$ and $h_i$ are not null (for simplicity in this lemma we are using the additive notation for the operation of $\cG$). Let us write $g_{i_1}, \ldots, g_{i_n}$ and $h_{j_1}, \ldots, h_{j_m}$ for the non-null components of $g$ and $h$. We can always find elements $\phi_g, \phi_h \in \Sp(\cG)$ such that
\[ \phi_g(g) = (\underbrace{(x, x), (x, x), \ldots, (x, x)}_{n \text{ times}}, (0,0), \ldots) \]
and
\[ \phi_h(h) = (\underbrace{(x, x), (x, x), \ldots, (x, x)}_{m \text{ times}}, (0,0), \ldots), \]
for an arbitrary fixed non null-element $x$ of $\F_p$. This is due to the fact that in a direct sum of the form $(\F_p^2, \sigma_2) \oplus (\F_p^2, \sigma_2)$ the map that sends an element $(x_1, x_2, x_3, x_4)$ to $(x_3, x_4, x_1, x_2)$ is easily seen to be in the symplectic group. The general result follows by induction.\\
 In this way, we can map $g$ and $h$ to vectors whose non-null components are in the first $n$ and $m$ coordinates respectively. Then, the symplectic group $\Sp((\F_p^2, \sigma_2))$ of each factor maps injectively in $\Sp(\cG)$ because they are split hyperbolic plane in the sense of Definition \ref{defn:split_hyperbolic_plane}. Then, applying Lemma \ref{lem:Orbita_su_F_p} we get the desired form for $\phi_g(g)$ and $\phi_h(h)$.

It remains to check that we can always map the elements $\phi_g(g)$ and $\phi_h(h)$ to each other via an element of $\Sp(\cG)$. The general case can be reduced to the case $\phi_g(g) = ((x, x), (x, x))$ and $\phi_h(h) = ((x, x), (0,0))$ by induction, so we discuss only this case. Then, it is enough to show that there exists a symplectic automorphism $\psi: (\F_p^4, \sigma_4) \to (\F_p^4, \sigma_4)$ such that  $\psi(\phi_g(g)) = \phi_h(h)$. Consider the following matrix
\[ \bM =
\begin{pmatrix}
1& 0 & 0 & 0 \\
0 & 1 & 1 & -1 \\
1 & 0 & 1 & 0 \\
1 & 0 & 0 & 1
\end{pmatrix}.
 \]
It is easy to check that the automorphism induced by $\bM$ on $(\F_p^4, \sigma_4)$ belongs to the symplectic group. Since we have
\[
\begin{pmatrix}
1 & 0 & 0 & 0 \\
0 & 1 & 1 & -1 \\
1 & 0 & 1 & 0 \\
1 & 0 & 0 & 1
\end{pmatrix}
\begin{pmatrix}
x \\
x \\
0 \\
0
\end{pmatrix}
= 
\begin{pmatrix}
x \\
x \\
x \\
x
\end{pmatrix}
\]
we can conclude the proof.
\end{proof}

As immediate consequence of Lemma \ref{lemma:infinite_torsion_orbits}  we obtain that a $\Sp(\cG)$-invariant state on $\cW_\cG$ must be constant on all generators. With the next theorem prove uniqueness of $\Sp(\cG)$-invariant states on $\cW_\cG$.

\begin{thm} \label{thm:infinite_torsion}
Let $\cW_\cG$ be the symplectic twisted group algebra associated to $\cG = \bigoplus_{i \in I} (\F_p^2, \sigma_2)$ and consider the constant normalized functional $\omega: \cW_\cG \to \C$ given by  
\[ \omega(g) = \begin{cases}
1 & \text{ if } g = 1_\cG \\
q \in [-1, 1], \text{ otherwise. } 
\end{cases}
 \]
Then, $\omega$ is positive if and only if $q = 0$.
\end{thm}
Before entering into the details of the proof, we define some useful notation.
\begin{notation}\label{eq:M_notation}
 With  $\bM_n(p, q)$, we denote the $n \times n$ matrix given by
\begin{align*} 
    (\bM_n(p, q))_{j,j} &= 1 &   j \geq 1\\
    (\bM_n(p, q))_{1,j}&=(\bM_n(p, q))_{j,1} = p &   j > 1\\
    (\bM_n(p, q))_{i,j} &= (\bM_n(p, q))_{j,i} = q &  \text{else}
    \end{align*} 
for which $p, q \in \C$. 
\end{notation}
\begin{proof}[Proof of Theorem~\ref{thm:infinite_torsion}]
Consider any element in $g \in \cG$. We want to find, for any $n \in \N$, suitable sequences of elements  $\aa_{1, k}, \ldots, \aa_{n p, k} \in \cW_\cG$, for $1 \le k \le {p}$ such that the matrix associated to the quadratic for $\omega(\aa_{i,k}^* \aa_{j,k})$, for $1 \le i {, j} \le n p$, is of the form $\bM_{n p}(q, q e^{2 \pi \i \frac{k}{p}})$, using the notation of Notation~\ref{eq:M_notation}. If $q$ is such that the functional $\omega$ is positive, then all matrices $\bM_{n p}(q, q e^{2 \pi \i \frac{k}{p}})$ must be positive.

Suppose that such elements exist. Then, we can consider the matrix
\[ \bR_n = \frac{1}{p} \sum_{k = 1}^p     \bM_{n p}(q, q e^{2 \pi \i \frac{k}{p}}) = \bM_{n p}(q, 0). \]
By hypothesis this can be done for any $n$. Hence, for a fixed $q$ there exists a $n$ big enough such that the matrix $\bR_n$ is non-positive, as it follows immediately by computing the determinant of $\bR_n$. But this is in contradiction with the hypothesis that all the matrices $\bM_{pn}(q, q e^{2 \pi \i \frac{k}{p}})$ are positive, as the convex sum of positive matrices must be a positive matrix.

It remains to show that it is always possible to find such elements $\aa_{i, k}$, for which the matrix associated to the quadratic for $\omega(\aa_{i,k}^* \aa_{j,k})$ is $\bM_{n p}(q, q e^{2 \pi \i \frac{k}{p}})$. To this end, consider a generic element $e_1 \in \cG$. It is enough to find for any $n \in \N$ elements $e_2, \ldots, e_n \in \cG$ such that
\begin{equation} \label{eq:sigma_system}
\sigma(e_i, e_j) = e^{2 \pi \i k/p }. 
\end{equation}
But since $e^{2 \pi \i k/p }$ is a primitive $p$-th root of $1$, we can identify the elements $e^{2 \pi \i k/p }$ with elements of $\F_p$ and the \eqref{eq:sigma_system} becomes a linear system of finitely many equations with $\F_p$ coefficients. Since $\F_p$ is a field and $\cG$ is an infinite dimensional vector space over it we can always find a suitable finite dimensional subspace of $\cG$ where the system \eqref{eq:sigma_system} admits solutions, concluding the proof.
\end{proof}

We conclude this section, by showing how the proof of Theorem \ref{thm:infinite_torsion} can be generalized to the case when $\F_p$ is replaced with $\Z/n\Z$. Indeed, the case when $n = p^f$, and therefore $(\Z/n\Z)^2$ is equipped with is canonical $\T$-symplectic form determined by $\sigma_2((1,0), (0,1)) = e^{\frac{2 \pi \i}{n}}$, can be dealt in the same way as done in Theorem \ref{thm:infinite_torsion}, using the fact that in this case there are $f$ different orbits each of which with infinite element, besides the trivial orbit of the identity element. Therefore the following corollary.

\begin{cor} \label{cor:infinite_torsion}
Theorem \ref{thm:infinite_torsion} remains true when for $\cG = \bigoplus_{i \in I} ((\Z/n\Z)^2, \sigma_2)$.
\end{cor}
\begin{proof}
We already remarked that the case when $n = p^f$ is similar to the case of the theorem. The general case follows from Corollary \ref{cor:direct_sum_torsion} and the observation that in this case 
\[ \Sp(\bigoplus_{p \in \P} \cG_p) \cong \prod_{p \in \P} \Sp(\cG_p) \]
where $\cG_p$ is the $p$-primary part of $\cG$.
\end{proof}

 We conclude this section by remarking that it is easy to find examples of infinite torsion groups on which the action of any group of automorphisms cannot give ergodic \stg-algebras, as their $p$-primary parts are finite, \eg $\bigoplus_{p \in \P} \F_p$.

\section{Conclusions} \label{sec:conclusion}
We conclude this paper with some conjectures and an application of our results to abelian Chern-Simons theory. To this end, let us remark that, in all cases we have considered so far, we proved the uniqueness of $\Sp(\cG)$-invariant states when the group $\cG$ was equipped with a non-degenerate pre-symplectic form. Indeed, it is easy to see that this is a necessary condition  on $\cW_\cG$ for having a unique $\Sp(\cG)$-invariant state. 
\begin{cor}[Corollary to Proposition~\ref{thm:CTori}.]\label{cor:degenerate sympl. form}
Let $(\cG,\sigma)$ be a $\T$-pre-symplectic abelian group. If $\sigma$ is degenerate, then $\cW_\cG$ admits plenty of $\Sp(\cG)$-invariant states.
\end{cor}
\begin{proof}
If the pre-symplectic form on $\cG$ is degenerate, then $\cG^\perp \ne 0$ and $\Phi(\cG^\perp) \subset \cG^\perp$ for all $\Phi \in \Sp(\cG)$. Then, it is easy to check that any functional defined by
$$
\tau(g)=\begin{cases} 1 & \text{ if } g=1_\cG \\ 
0 & g \not\in \cG^\perp \\
q & g \in \cG^\perp  \end{cases}
$$
with $q \in (0, 1)$, is $\Sp(\cG)$-invariant and positive on account of the computations done in the proof of Proposition \ref{thm:CTori}.
\end{proof}
The next theorem summarizes our main results on the uniqueness of $\Sp(\cG)$-invariant states, \ie Corollary~\ref{cor:degenerate sympl. form}, Theorem~\ref{thm:NCtori_no_torsion}, Theorem~\ref{thm:infinite_torsion} and Corollary~\ref{cor:infinite_torsion}.

\begin{thm} \label{thm:sum_up}
Let $(\cG, \sigma)$ be a $\T$-pre-symplectic abelian group, then
\begin{enumerate}
\item if $\sigma$ is degenerate, then $\cW_\cG$ admits plenty of $\Sp(\cG)$-invariant states;
\item if $(\cG, \sigma)$ is symplectic, irrational (in the sense of Definition \ref{defn:irrational}), the symplectic form is diagonalizable (in the sense of Notation \ref{not: diagonalization}) and ${\Sp(\cG)\otimes_\Z \Q} \supset \Sp^H(\cG \otimes_\Z \Q)$, then $\cG$ is torsion-free and $\cW_\cG$ admits only one invariant state;
\item if $(\cG, \sigma) \cong \bigoplus_{i \in I} ((\Z/n\Z)^2, \sigma_2)$, where $I$ has infinite cardinality, then the associated \stg-algebra admits a unique $\Sp(\cG)$-invariant state.
\end{enumerate}
\end{thm}

From these results, we conjecture the following to hold.

\begin{conj} \label{conj:symplectic_ergodic}
If $(\cG, \sigma)$ is a $\T$-symplectic abelian group such that the action of $\Sp(\cG)$ on $\cW_\cG$ is ergodic, then there are no non-trivial $\Sp(\cG)$-invariant states on $\cW_\cG$.
\end{conj}

A step in this direction is the following weak version of Conjecture \ref{conj:symplectic_ergodic}.

\begin{conj} \label{conj:symplectic_ergodic_weak}
If $(\cG, \sigma)$ is a irrational $\T$-symplectic abelian group such that the action of $\Sp(\cG)$ on $\cW_\cG$ is ergodic and $\sigma$ is diagonalizable, then there are no non-trivial $\Sp(\cG)$-invariant states on $\cW_\cG$.
\end{conj}

One possible approach to Conjecture \ref{conj:symplectic_ergodic_weak} is to prove that the technical assumption ${\Sp(\cG) \otimes_\Z \Q} \supset \Sp^H(\cG \otimes_\Z \Q)$ in Theorem \ref{thm:NCtori_no_torsion} is always satisfied. 

\vspace{0.5cm}

We conclude this paper with an application of the results proven so far inspired by \cite{DMS}. In \lc, the quantization of Abelian Chern-Simons theory is interpreted as a functor 
$$\mathfrak{A}: \Man_2 \to\stgAlg $$
which assigns symplectic twisted group $*$-algebras to 3-dimensional manifolds of
the form $\R \times \Sigma$, where $\Sigma$ is a 2-dimensional oriented manifold.
This assignment is obtained by composing the functor $$ \mathfrak{S}:=( H_c^1(-,\Z),\sigma): \Man_2 \to \bP\bSymp_\T $$ 
which assign to every $\Sigma\in\Man_2$ its first (singular) homology group with compact support $H_c^1(\Sigma)$ endowed with a $\T$-valued pre-symplectic form $\sigma$, with the functor
$$ \mathfrak{CCR}: \bP\bSymp_\T \to \stgAlg$$
which assign to every $\Z$-pre-symplectic abelian group a symplectic twisted group algebra. We say that such a functor $\mathfrak{A}$ is a \emph{Chern-Simons functor}.
 
By functoriality, $\mathfrak{CCR}$ induces a representation of the group of orientation preserving diffeomorphisms Diff$^+ (\Sigma)$ of $\Sigma$, that is the group of automorphisms of $\Sigma$ in $\Man_2$, on $\mathfrak{A}(\Sigma)$ as a group of $*$-algebra automorphisms.
This representation can be related to a representation of the mapping class group and, for compact $\Sigma$, to a
representation of the discrete symplectic group $\Sp(2n,\Z)$. \\
For a quantum physical interpretation of quantum Abelian Chern-Simons theory it is necessary to choose for each $\Sigma\in\Man_2$ a state $\omega_\Sigma: \mathfrak{A}(\Sigma) \to \C$ on the \stg-algebra $\mathfrak{A}(\Sigma)$. Motivated by the functoriality of the association, it seems natural to demand that the family of states $\{\omega_\Sigma\}_{\Sigma\in\Man_2}$ is compatible with
the functor $\mathfrak{A}: \Man_2 \to \stgAlg$ in the sense that
$$\omega_{\Sigma'} \circ \mathfrak{A}(f) = \omega_\Sigma $$
for all $\Man_2$-morphisms $f : \Sigma \to \Sigma'$. Such compatible families of states are called \textit{natural states} on $\mathfrak{A}: \Man_2 \to \stgAlg.$
Even though the idea of natural states is very beautiful and appealing, there are hard obstructions to the existence of natural states.

Using the results in Theorem~\ref{thm:sum_up} we can now generalize the main result of \cite{DMS}. Before stating the theorem, let us given a definition.
\begin{defn}
We say that a Chern-Simons functor is \emph{irrational and not degenerate} if each $\mathfrak{A}(\Sigma)$ is irrational in the sense of Definition \ref{defn:irrational} and, for any $\Sigma$ such that $H^1_c(\Sigma;\Z)\simeq \Z^{2n}$, the symplectic form is \emph{non-degenerate.}
\end{defn}

\begin{thm}
 There exists no natural state for an irrational and non-degenerate Chern-Simons functor.
\end{thm}
\begin{proof}
Let us assume that there exists a natural state $\{\omega_{\Sigma}\}_{\Sigma\in\Man_2}$. Consider the $\Man_2$-diagram 
$$
 \bS^2 \stackrel{f_1}{\longleftarrow} \R\times \T\stackrel{f_2}{\longrightarrow} \T^2
$$
describing an orientation preserving open embedding of the cylinder $\R\times\T$ into the $2$-sphere $\bS^2$ and 
the $2$-torus $\T^2$. The Chern-Simons functor assigns $*$-homomorphisms
$$\mathfrak{A}(\bS^2) \stackrel{\mathfrak{A}(f_1)}{\longleftarrow} \mathfrak{A}(\R\times \T)
\stackrel{\mathfrak{A}(f_2)}{\longrightarrow} \mathfrak{A}(\T^2)$$
and the naturality of the state implies the condition
\begin{equation}\label{eqn:statecondition}
\omega_{\bS^2}^{}\circ \mathfrak{A}(f_1) = \omega_{\R\times\T}^{} = \omega_{\T^2}^{}\circ \mathfrak{A}(f_2)~.
\end{equation}
Because of $H^1_{c}(\bS^2;\Z)=0$, it follows that $\mathfrak{A}(\bS^2) \simeq \C$
and hence $\omega_{\bS^2}^{} = \Id_{\C}$ has to be the unique state on $\C$.
Using further that $H^1_c(\R\times \T) \simeq \Z$, it follows
by the first equality in \eqref{eqn:statecondition}
\begin{equation}\label{eqn:condition1}
\omega_{\R\times\T}^{}\big(n\big)= 1~,
\end{equation}
for all $n\in\Z$. By the non-degeneracy hypothesis on $\mathfrak{A}$ it follows that $(H^1_{c}(\T^2;\Z),\sigma_{\T^2})$ is isomorphic to the abelian group $\Z^2$ with an irrational $\T$-symplectic form, \ie $\mathfrak{A}(\T^2)$ is an algebraic irrational noncommutative torus.
We can choose $f_2$ such that the $\ast$-algebra homomorphism $\mathfrak{A}(f_2) : \mathfrak{A}(\R\times\T)\to \mathfrak{A}(\T^2)$ is given by $n \mapsto {(n,0)}$, for all $n\in \Z$.
As a consequence of \eqref{eqn:statecondition} and \eqref{eqn:condition1}, we obtain that
$$
\omega_{\T^2}^{}\big({(n,0)}\big) =1~,
$$
for all $n\in\Z$, which is not positive by Theorem \ref{thm:NCtori_no_torsion_CD} and hence not a state.
\end{proof}

\vspace{1cm}
{

}

\end{document}